\documentclass[a4paper,11pt]{article}
\usepackage{mfpaperstuff2}
\usepackage{mfabacus}
\usepackage{multirow}
\usepackage{yfonts}
\usepackage{tikz-cd}
\usepackage{verbatim}
\usepackage{caption}
\usetikzlibrary{calc}

\begin{document}
\newcommand\jdoms\blacktriangleright
\makeatletter
\newcommand*{\jdom}{%
\mathrel{%
\mathpalette\@blackgeq{}%
}%
}
\newcommand*{\@blackgeq}[1]{%
% #1: math style
\vcentre{%
\m@th
\setbox0=\hbox{$#1\mkern3mu$}%
\setbox2=\hbox{$#1\vcentre{}$}%
\setbox4=\hbox{\raisebox{1.1pt}[.2pt][.2pt]{$#1\geqslant$}}%
\hbox{$#1\blacktriangleright$}%
\nointerlineskip
\kern\wd0 %
\copy4 %
}6%
}
\makeatother

\newcommand\myblue{blue!50!white}
\newcommand\ee{\mathrm e}
\newcommand\luar[3]{\draw[->,#1](#2)to[out=165,in=15](#3);}
\newcommand\ruar[3]{\draw[->,#1](#2)to[out=15,in=165](#3);}
\newcommand\ldar[3]{\draw[->,#1](#2)to[out=195,in=345](#3);}
\newcommand\rdar[3]{\draw[->,#1](#2)to[out=345,in=195](#3);}
\newcommand\Ga\Gamma
\newcommand\reg{\operatorname{reg}}
\newcommand\grib{maximal \conf}
\newcommand\gribs{\grib{}s\xspace}
\newcommand\wht{\Yfillcolour{white}}
\newcommand\Hom{\operatorname{Hom}}
\renewcommand\rt[1]{\rotatebox{90}{$#1$}}
\newcommand\zpz{\bbz/p\bbz}
\newcommand\zez{\bbz/e\bbz}
\newcommand\ebeb{e_B\,e_{B'}}
\newcommand\spe[1]{\operatorname{S}^{#1}}
\newcommand\smp[1]{\operatorname{D}^{#1}}
\newcommand\dspe[1]{\operatorname{S}'(#1)}
\newcommand\dsmp[1]{\operatorname{D}'(#1)}
\newcommand\dn[2]{[\spe{#1}:\smp{#2}]}
\newcommand\ol[1]{\widebar{#1}}
\newcommand\FS[1]{\bbfS_{#1}}
\newcommand\kS[1]{kS_{#1}}
\newcommand\daha{\calh_n^k}
\newcommand\cenalg[3]{\calc_{#1,#2}^{#3}}
\newcommand\cenalgA{\cenalg lmA}
\newcommand\cenalgR{\cenalg lmR}
\newcommand\cenalgF{\cenalg lm{\bbf}}
\newcommand\cenalgk{\cenalg lmk}
\newcommand\skewp{skew-partition\xspace}
\newcommand\skewps{skew-partitions\xspace}
\newcommand\skewd{skew-diagram\xspace}
\newcommand\skewds{skew-diagrams\xspace}
\newcommand\sk[2]{#1\setminus#2}
\newcommand\abac{abacus configuration\xspace}
\newcommand\abacs{\abac{}s\xspace}
\newcommand\Ab[1]{\operatorname{Ab}(#1)}
\newcommand\conf{$p$-shape\xspace}
\newcommand\confs{$p$-shapes\xspace}
\newcommand\X{\mu_X\setminus\la_X}
\newcommand\Y{\mu_Y\setminus\la_Y}
\newcommand\Xmod{\mu_Z\setminus\la_X}
\newcommand\Ymod{\mu_Y\setminus\la_W}
\newcommand\conjbl{combinatorial block\xspace}
\newcommand\conjbls{combinatorial blocks\xspace}
\newcommand\grb{generalised ribbon block\xspace}
\newcommand\Grb{Generalised ribbon block\xspace}
\newcommand\grbs{\grb{}s\xspace}
\newcommand\Grbs{\Grb{}s\xspace}
\newcommand\ribbl{ribbon block\xspace}
\newcommand\ribbls{ribbon blocks\xspace}
\newcommand\belbl{belt block\xspace}
\newcommand\belbls{belt blocks\xspace}
\newcommand\rib{ribbon\xspace}
\newcommand\ribs{ribbons\xspace}
\newcommand\Ribs{Ribbons\xspace}
\newcommand\belt{belt\xspace}
\newcommand\belts{belts\xspace}
\newcommand\reft{refinement\xspace}
\newcommand\refts{refinements\xspace}
\newcommand\Dist[2]{\operatorname{D}(#1,#2)}
\newcommand\dist[2]{\operatorname{d}(#1,#2)}
\newcommand\arrg{arrow graph\xspace}
\newcommand\arrgs{arrow graphs\xspace}
\newcommand\nodres{node of residue\xspace}
\newcommand\forch{formal character\xspace}
\newcommand\forchs{formal characters\xspace}
\newcommand\cench{central character\xspace}
\newcommand\ch[1]{\operatorname{\mathrm{ch}}(#1)}
\newcommand\cch[1]{\operatorname{\calc}(#1)}
\newcommand\arr[1]{\Ga_{#1}}
\newcommand\arrX{\arr X}
\newcommand\arrY{\arr Y}
\newcommand\conncomp{connected component\xspace}
\newcommand\conncomps{\conncomp{}s\xspace}
\newcommand\resp[1]{\operatorname{cont}_p(#1)}
\newcommand\mulam{\mu\setminus\la}
\newcommand\mulamone{\mu_1\setminus\la_1}
\newcommand\mulamtwo{\mu_2\setminus\la_2}
\newcommand\core[1]{\operatorname{core}_p(#1)}
\newcommand\corep[2]{\operatorname{core}_{#1}(#2)}

\newcommand\abx{\Ab X}
\newcommand\aby{\Ab Y}

\newcommand\wb{1.35}

\newcommand\ipo{\text{\footnotesize$i{+}1$}}
\newcommand\ipr{\text{\footnotesize$i{+}r$}}
\newcommand\ipj{\text{\footnotesize$i{+}j$}}
\newcommand\imo{\text{\footnotesize$i{-}1$}}
\newcommand\lmo{\text{\footnotesize$l{-}1$}}
\newcommand\kpo{\text{\footnotesize$k{+}1$}}
\newcommand\jmo{\text{\footnotesize$j{-}1$}}
\newcommand\jpo{\text{\footnotesize$j{+}1$}}
\newcommand\jo{j_1}
\newcommand\jt{j_t}
\newcommand\bxi{\young(i)}
\newcommand\bxj{\young(j)}
\newcommand\bxk{\young(k)}
\newcommand\bxl{\young(l)}
\newcommand\bxipo{\gyoung(_\wb\ipo)}
\newcommand\bxipr{\gyoung(_\wb\ipr)}
\newcommand\bxipj{\gyoung(_\wb\ipj)}
\newcommand\bxkpo{\gyoung(_\wb\kpo)}
\newcommand\bxjmo{\gyoung(_\wb\jmo)}
\newcommand\bxjpo{\gyoung(_\wb\jpo)}
\newcommand\bxlmo{\gyoung(_\wb\lmo)}

\newcommand\mattnote[1]{\textcolor{\myblue!60!green}{\textbf{Matt:} #1}}
\newcommand\putinote[1]{\textcolor{red}{\textbf{Puti:} #1}}
\newcommand\ball{.3}
\newcommand\spac{.15}
\newcommand\yoy{1}
\newcommand\spce[2]{\draw(#1,#2*\yoy-.5*\yoy)--++(0,\yoy);\draw(#1-.5*\spac,#2*\yoy)--++(\spac,0);}
\newcommand\bll[2]{\draw(#1,#2-.5*\yoy)--++(0,\yoy);\shade[shading=ball,ball color=black](#1,#2*\yoy)circle(\ball);}
\newcommand\lfbll[2]{\draw(#1,#2-.5*\yoy)--++(0,\yoy);\begin{scope}\clip(#1,#2*\yoy-.5*\yoy)rectangle++(-\ball,\yoy);\shade[shading=ball,ball color=black](#1,#2*\yoy)circle(\ball);\end{scope}}
\newcommand\redlfbll[2]{\draw(#1,#2-.5*\yoy)--++(0,\yoy);\begin{scope}\clip(#1,#2*\yoy-.5*\yoy)rectangle++(-\ball,\yoy);\shade[shading=ball,ball color=red](#1,#2*\yoy)circle(\ball);\end{scope}}
\newcommand\bluelfbll[2]{\draw(#1,#2-.5*\yoy)--++(0,\yoy);\begin{scope}\clip(#1,#2*\yoy-.5*\yoy)rectangle++(-\ball,\yoy);\shade[shading=ball,ball color=\myblue](#1,#2*\yoy)circle(\ball);\end{scope}}
\newcommand\redrtbll[2]{\draw(#1,#2-.5*\yoy)--++(0,\yoy);\begin{scope}\clip(#1,#2*\yoy-.5*\yoy)rectangle++(\ball,\yoy);\shade[shading=ball,ball color=red](#1,#2*\yoy)circle(\ball);\end{scope}}
\newcommand\bluertbll[2]{\draw(#1,#2-.5*\yoy)--++(0,\yoy);\begin{scope}\clip(#1,#2*\yoy-.5*\yoy)rectangle++(\ball,\yoy);\shade[shading=ball,ball color=\myblue](#1,#2*\yoy)circle(\ball);\end{scope}}
\newcommand\rtbll[2]{\draw(#1,#2-.5*\yoy)--++(0,\yoy);\begin{scope}\clip(#1,#2*\yoy-.5*\yoy)rectangle++(\ball,\yoy);\shade[shading=ball,ball color=black](#1,#2*\yoy)circle(\ball);\end{scope}}

\title{Ribbon blocks for centraliser algebras\\of symmetric groups}


\author{\begin{tabular}{cc}{\Large Matthew Fayers}&{\Large Lorenzo Putignano}\\[3pt]
{\normalsize Queen Mary University of London}&{\normalsize Universit\`a degli Studi di Firenze}\\
{\normalsize\texttt{\normalsize m.fayers@qmul.ac.uk}}&{\normalsize\texttt{\normalsize lorenzo.putignano@unifi.it}}
\end{tabular}}

%\author{\begin{tabular}{cc}{\Large xxxx}&{\Large yyyy}\\[3pt]
	%	{\normalsize xxxx}&{\normalsize yyyy}\\
	%	{\normalsize\texttt{\normalsize xxxx}}&{\normalsize\texttt{\normalsize yyyy}}
%\end{tabular}}

\renewcommand\auth{Matthew Fayers \& Lorenzo Putignano}
\runninghead{Ribbon blocks for centraliser algebras of symmetric groups}

\toptitle

\begin{abstract}
Suppose $l,m$ are natural numbers with $l\ls m$, and $\bbf$ a field of characteristic $p$, and let $\cenalgF$ denote the centraliser of the group algebra $\bbf S_l$ inside $\bbf S_m$. Ellers and Murray give a conjectured classification of the blocks of $\cenalgF$, in terms of the $p$-blocks of $S_l$ and $S_m$. We prove this conjecture for a family of blocks that we call \emph{ribbon blocks} and \emph{belt blocks}. These are the blocks containing Specht modules labelled by \skewps having no repeated entries in their $p$-content.
\end{abstract}

\tableofcontents

\section{Introduction}

Let $G$ be a finite group, $H$ a subgroup of $G$ and $\bbf$ an algebraically closed field. The \emph{centraliser algebra} $\bbf G^H$ is defined by
\[\bbf G^H=\{a\in\bbf G\ |\ ah=ha\ \forall h\in H\}.\]

This algebra appears in a series of papers by Ellers \cite{el1,el2,el3,el4} mainly motivated by the attempt to extend some \emph{local-global} conjectures and theorems whenever $\bbf$ has positive characteristic. Such statements relate representation-theoretic information on $\bbf G^H$ to local information.  In particular, Ellers' main motivation lies in constructing a theory in which Alperin's Weight Conjecture and Brauer's First Main Thoerem are just special cases of a more general setting.

With this in mind, the centraliser algebra has become an important object of study. If $\bbf$ has characteristic zero, then a more general argument by Curtis and Reiner \cite[Section 11D]{currei} shows that $\bbf G^H$ is a semisimple algebra, with a straightforward construction of its simple modules. So the main focus is on the case where $\bbf$ has positive characteristic.

In this paper we focus on the case where both $G$ and $H$ are symmetric groups. Let $l,m$ be non-negative integers with $l\ls m$. The symmetric group $S_l$ is naturally a subgroup of $S_m$ and hence it makes sense to consider the centraliser algebra $\bbf S_m^{S_l}\vcentcolon=\cenalgF$. The main goal is to find a suitable labelling for the simple $\cenalgF$-modules and a complete description of its blocks, i.e.\ minimal two-sided ideals whose direct sum gives the whole algebra. A powerful result would be to give answers to these issues in terms of the combinatorics of integer partitions, somehow extending the symmetric group case pioneered by James \cite{james} which we recall in \cref{backnot}. Following \cite{klbook}, a complete set of pairwise non-isomorphic simple $\cenalgF$-modules when $\operatorname{char}(\bbf)=0$ is given by modules of the form \[\spe{\mulam}\vcentcolon=\Hom_{\bbf S_l}(\spe{\mu}\downarrow_{S_l},\spe{\la})\] where $\mu$ and $\la$ denote partitions of $m$ and $l$, respectively, and $\mulam$ is a \emph{\skewp} in the sense of \cref{secskew}. In view of the similarity with the symmetric group case, we name these modules \emph{Specht modules}. The Specht modules are also defined in positive characteristic, where they are generally reducible, and (unlike in the symmetric group case) there does not seem to be a straightforward way to obtian the simple modules from the Specht modules. In positive characteristic, a full description of the simple modules of $\cenalgF$ is far out of reach.

The goal of the present paper is to characterise certain blocks of $\cenalgF$. As in the symmetric group case, describing the blocks amounts to finding the appropriate subdivision of the set of Specht modules (or equivalently, of the set of \skewps).  In \cite{elmu}, Ellers and Murray suggested a classification of the blocks in positive characteristic and checked it for $m-l\le3$. Their conjecture (\cref{mainconj}) is inspired by the Nakayama Conjecture (\cref{nakconj}) and hence is in terms of cores of partitions. We define a \emph{\conjbl} to be a set of \skewps which is predicted to comprise a single block by \cref{mainconj}. The work of Ellers and Murray shows that every combinatorial block is a union of blocks, so the remaining task to prove their conjecture is to show that the Specht modules labelled by two \skewps in the same combinatorial block really do lie in the same block.

In this paper we prove Ellers and Murray's conjecture for a class of \conjbls of $\cenalgF$ we call \emph{ribbon blocks} and \emph{belt blocks}. These are precisely the \conjbls containin \skewps having no repetitions in their content (see \cref{secskew}). In \cref{rbsec} we show that (a proper approximation of) the decomposition matrix of a ribbon or belt block is a connected matrix; this implies that \conjbls of this kind are blocks of the centraliser algebra.

\begin{acks}
	The second author is supported by GNSAGA (INdAM) and has been funded by the European Union -- Next Generation EU, Missione 4 Componente 1, PRIN 2022-2022PSTWLB -- Group Theory and Applications, CUP B53D23009410006.
\end{acks}

\section{Background and notation}\label{backnot}

In this chapter we introduce all the background we need in \cref{rbsec}. We begin with some basic notations. Let $R$ be a commutative ring with field of fractions $\bbf$. Let $\textfrak{p}$ be a maximal ideal of $R$ and denote by $k$ the residue field $R/\textfrak{p}$ of positive characteristic $p$. We suppose that both $\bbf$ and $k$ are algebraically closed.

If $e\gs2$ is an integer, then we write $\widebar a$ to mean $a+e\bbz$, for $a\in\bbz$, so that $\zez=\lset{\widebar a}{a\in\bbz}=\{\widebar0,\widebar1,\dots,\widebar{e-1}\}$.% We totally order $\zez$ via $\widebar0<\widebar1<\dots<\widebar{e-1}$. So it makes sense to talk about the min/max of a subset of $\zez$. \textcolor{red}{probably unuseful because we aren't considering the max anymore}

In this paper, $\mathbb{N}$ denotes the set of positive integers.

\subsection{Partitions and Specht modules}\label{secpart}

In this section we introduce the main combinatorial objects in this paper. A \emph{partition} is a non-increasing sequence $\nu=(\nu_1,\nu_2,\dots)$ of non-negative integers such that $\nu_N=0$ for $N\gg0$. The integer $\nu_k$ is called the $k$-th \emph{part} of $\nu$ and the number of non-zero parts is referred as the \emph{length} of $\nu$ and denoted by $l(\nu)$. The (finite) sum $\nu_1+\dots+\nu_{l(\nu)}$ is referred as the \emph{size} of $\nu$ and denoted by $|\nu|$. If $n\ge0$ and $\nu$ is a partition of size $n$, we say that \emph{$\nu$ is a partition of $n$} and we write $\nu\vdash n$. We denote by $\vn$ the unique partition of 0. When writing partitions we omit the trailing zeroes and group together equal parts with a superscript, e.g.\ $(3,2^3,1)$ stands for the length $5$ partition $(3,2,2,2,1,0,\dots)\vdash10$. We denote by $\calp$ the set of all partitions and by $\calp_n$ the set of partitions of a fixed non-negative integer~$n$. 

The set of partitions $\calp_n$ plays the role of a powerful labelling set for a relevant class of $RS_{n}$, $\bbf S_n$ and $kS_n$-modules known in the literature as \emph{Specht modules}. The set of Specht modules $\{\spe{\nu}\ |\ \nu\in\calp_n\}$ is a crucial object in the study of the representation theory of the symmetric group $S_n$ because it provides a complete set of pairwise non-isomorphic simple $\bbf S{n}$-modules (in fact, more generally, over any characteristic zero field). If $\bbf$ is replaced by the residue field $k$, an analogous statement holds if $n<p$. When $n\gs p$, a complete family of pairwise non-isomorphic simple $\kS{n}$-modules can be constructed from the Specht modules: if $\mu$ is a \emph{$p$-regular} partition (meaning that it does not have $p$ equal positive parts), then the Specht module $\spe\mu$ over $k$ has a unique simple quotient $\smp\mu$, and the simple modules obtained in this way give a complete set of non-isomorphic simple $kS_n$-modules. The main outstanding problem in the modular representation theory of the symmetric group is the determination of the \emph{decomposition numbers} $[\spe\la:\smp\mu]$ for all partitions $\la,\mu$ of $n$ with $\mu$ $p$-regular. This problem goes back to the work of Robinson in the 1960s \cite{robin}, but remains unsolved in general.
% is obtained looking at the heads of those Specht modules whose label is a \emph{$p$-regular partition} (see \cite{james} for a definition of $p$-regular partition and further details). We write $\calp_n^{\reg}$ for the set of $p$-regular partitions of $n$ and $\smp{\mu}$, $\mu\in\calp_n^{\reg}$, for a simple $\kS{n}$-module. With these two labellings at hand, the \emph{decomposition matrix of $S_n$} is the matrix \[
% D_n=(d_{\la\mu})_{\la\in\calp_n,\ \mu\in\calp_n^{\reg}}
% \] where $d_{\la\mu}$ is the multiplicity of the simple $\kS{n}$-module $\smp{\mu}$ as a composition factor of the Specht $\kS{n}$-module $\spe{\la}$. The evaluation of the \emph{decomposition numbers} is a largely studied topic in algebraic combinatorics and very deep results have been obtained in this direction; for example Lascoux-Leclerc-Thibon gave an algorithm \textcolor{red}{scrivere meglio e aggiungere referenza e magari qualche altro esempio}. Decomposition matrices analogue to $D_n$ will be central later in the chapter.\mattnote{I think we should abbreviate this -- we don't need notation for simple modules etc.}

\subsection{Young diagrams and standard tableaux}\label{secyoung}

We now introduce a good method for visualising partitions. The \emph{Young diagram} of a partition $\nu$ is the set 
\[
[\nu]=\{(r,c)\in\mathbb{N}^2\ |\ c\le\nu_r\}.
\]
The elements of $[\nu]$, and more generally of $\mathbb{N}^2$, are called \emph{nodes}. We draw a Young diagram as an array of boxes in the plane such that the horizontal axis is oriented left-to-right and the vertical axis top-to-bottom (commonly referred as the \emph{English} notation). For example, $[(3,2^3,1)]$ appears as follows.
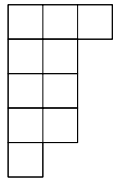
\[
\yng(3,2,2,2,1)
\]
Motivated by the drawing, we say that the node $(r,c)$ is \emph{above} the node $(s,d)$, or that $(s,d)$ is \emph{below} $(r,c)$, if $r<s$. Analogously, we say that $(r,c)$ lies \emph{to the left} of $(s,d)$, or that $(s,d)$ lies \emph{to the right} of $(r,c)$, whenever $c<d$. 

Fix $e$ is a positive integer. The \emph{content} of a node $(r,c)$ is the integer $c-r$, and the \emph{$e$-residue} of $(r,c)$ is $\widebar{c-r}$. Note that nodes lying in the same diagonal of $\mathbb{N}^2$ have the same content and therefore the same $e$-residue for all $e$. If $\nu\in\calp$, we define the \emph{$e$-content} and we write $\mathrm{cont}_e(\nu)$, for the multiset of $e$-residues of the nodes in $[\nu]$. For example,
\[
\mathrm{cont}_3(3,2^3,1)=\{\widebar{-4},\widebar{-3},\widebar{-2},\widebar{-2},\widebar{-1},\widebar{-1},\widebar0,\widebar0,\widebar1,\widebar2\}=\{\widebar0,\widebar0,\widebar0,\widebar1,\widebar1,\widebar1,\widebar2,\widebar2,\widebar2,\widebar2\}.
\]

The \emph{rim} of a partition $\nu$ is the set of nodes $(r,c)\in[\nu]$ such that $(r+1,c+1)\notin[\nu]$. We regard this set as ordered from from SW to NE saying that a node $(r,c)$ comes after $(s,d)$ if $(r,c)$ lies above or to the right of $(s,d)$. 
%Note that this coincides with ordering the nodes in the rim in increasing order of content. 
For $h\ge1$, we define a \emph{removable $h$-hook} of $\nu$ to be an interval $(r_1,c_1),\dots,(r_h,c_h)$ of the rim of $\nu$ such that $(r_1+1,c_1)\notin[\nu]$ and $(r_h,c_h+1)\notin[\nu]$ and . We call $(r_1,c_1)$ and $(r_h,c_h)$ the \emph{foot node} and the \emph{hand node} of the removable $h$-hook, respectively. By definition the nodes of a removable $h$-hook can be removed from $[\nu]$ to give the Young diagram of a partition of $|\nu|-h$. If $h=1$, a removable $1$-hook is simply called a \emph{removable node}. Dually, we define the \emph{neighbours} of $[\nu]$ to be the nodes $(s,d)\notin[\nu]$ such that either $s=1$ or $d=1$, or $(s-1,d-1)\in[\nu]$. The neighbours of $\nu$ are naturally ordered in increasing order of content. For $h\ge1$, an interval of neighbours $(s_1,d_1),\dots,(s_h,d_h)$ of $[\nu]$ is called an \emph{addable $h$-hook} if $d_1=1$ or $(s_1,d_1-1)\in[\nu]$, and $s_h=1$ or $(s_h-1,d_h)\in[\nu]$. As in the dual case, we call $(s_1,d_1)$ the \emph{foot node} and $(s_h,d_h)$ the \emph{hand node} of the addable $h$-hook. Again we can see that, pictorially, an addable $h$-hook is a series of adjacent nodes which can be added to $[\nu]$ to obtain the Young diagram of a partition of $|\nu|+h$. An addable $1$-hook is called an \emph{addable node}.

Fix $n\ge0$ and $\nu\in\calp_n$. We define a \emph{$\nu$-tableau} as a bijection $T:[\nu]\rightarrow\{1,\dots,n\}$ represented by filling the nodes of the Young diagram $[\nu]$ with their images under $T$. A $\nu$-tableau is called \emph{standard} if its entries increase along rows and down columns.
\begin{eg}
Let $\nu=(5,3,2,1)\vdash11$. The following are two $\nu$-tableaux, the one on the left being standard.
\[
\young(1378<10>,259,46,<11>)\ \ \ \ \ \ \ \ \ \ \ \ \young(4928<10>,<11>73,16,5)
\] 
\end{eg}

\begin{comment}
The class of standard $\nu$-tableaux plays a central role in the representetion theory of symmetric groups. These tableaux are in bijection with paths from the trivial Specht module $\spe{\vn}$ to the Specht module $\spe{\nu}$ in the branching graph of symmetric groups (referenza). Given a standard $\nu$-tableau $T$, the corresponding path is inductively found starting from the empty partition $\vn$ and adding the node $T^{-1}(k)$ at the $k$th step, for $k=1,\dots,n$. As a consequence we have that the number of standard $\nu$-tableaux gives the dimension of the Specht module $\spe{\nu}$ over any commutative ring; see \cite{verok} for further details and a non-traditional approach to the representation theory of symmetric groups. This number can be evaluated directly from the Young diagram $[\nu]$ using the outstanding \emph{Hook-Length Formula} by Frame, Robinson and Thrall. Different proofs of the formula have been given in the 70 years of its existence, see for example \cite{hooklength} for an elementary one.
\end{comment}

\subsection{Skew-partitions and skew Specht modules}\label{secskew}

This section gets the reader in touch with the combinatorial protagonist of the paper. Since the algebra of interest -- as we shall see in \cref{seccentalg} -- relies on two natural numbers, it seems natural to consider a relation of \emph{containment} between pairs of partitions.

Let $\la,\mu\in\calp$. We say that $\la$ \emph{lies inside} $\mu$ if $[\la]\subseteq[\mu]$. In this case we say that $\mulam$ is a \emph{\skewp}. We immediately see that, if $\mulam$ is a \skewp, then $|\mu|\ge|\la|$ with equality holding if and only if $\mu=\la$. We call the non-negative integer $|\mu|-|\la|$ the \emph{size} of $\mulam$. Given two non-negative integers $l\le m$, we set $\calp_{l,m}\vcentcolon=\{\mulam\ |\ \la\in\calp_l,\ \mu\in\calp_m\}$. 

It is useful to visualise \skewps as Young diagrams. For a \skewp $\mulam$, we define the \emph{\skewd} $[\mulam]$ as the set subset of $\mathbb{N}^2$ comprising nodes $[\mu]\setminus[\la]$. For example, if $\mu=(6,4^2,3,2)\vdash19$ and $\la=(3,2^3,1)\vdash10$, the \skewd $[\mulam]$ comprises the coloured nodes in the picture below.

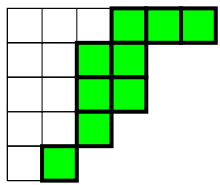
\[
\begin{tikzpicture}[scale=1]
\tyng(0cm,0cm,6,4,4,3,2)
\Ylinethick{1.5pt}
\Yfillcolour{green}
\tgyoung(0cm,0cm,:::;;;,::;;,::;;,::;,:;)
\end{tikzpicture}
\]
Note that two different \skewps can have the same Young diagram; for example, $[(3,2)\sm(3,1)]=[(2^2,1)\sm(2,1^2)]$.

We remark that if $\la$ lies inside $\mu$, the Young diagram $[\la]$ can be recovered by repeatedly deleting removable hooks from $[\mu]$. The union of these hooks comprises the \skewd $[\mulam]$. Note that this series of removals is not unique.

Similarly to the case of partitions, for $e$ a positive integer, we define the \emph{$e$-content} of a \skewp $\mulam$, and we write $\mathrm{cont}_e(\mulam)$, as the multiset of $e$-residues of the nodes in $[\mulam]$, e.g.
\[
\mathrm{cont}_4((6,4^2,3,2)\setminus(3,2^3,1))=\{\widebar{-3},\widebar{-1},\widebar0,\widebar1,\widebar1,\widebar2,\widebar3,\widebar4,\widebar5\}=\{\widebar0,\widebar0,\widebar1,\widebar1,\widebar1,\widebar1,\widebar2,\widebar3,\widebar3\}.
\]

In conclusion to this section, we generalise the ideas of tableau and standard tableau to the case of \skewps. If $\mulam$ is a \skewp of size $n$, a \emph{$\mulam$-tableau} is a bijection $T:[\mulam]\rightarrow\{1,\dots,n\}$ represented by filling the nodes of $[\mulam]$ with their images under $T$. A tableau is \emph{standard} if its entries increase along rows and down columns. Standard \emph{skew-tableaux} play a role in the representation theory of symmetric groups: if $\mulam\in\calp_{l,m}$, the set of standard $\mulam$-tableaux is in bijection with the paths between the Specht modules $\spe{\la}$ and $\spe{\mu}$ in the branching graph of symmetric groups (see for example \cite[Theorem 2.8.3]{sagan}). The branching rule itself also ensures that if $A$ is a commutative ring $A$, then the $A$-module
\[
\spe{\mulam}\vcentcolon=\Hom_{AS_l}(\spe{\la},\spe{\mu}\downarrow_{S_l})
\]
is non-trivial \iff $\mulam$ is a \skewp, in which the dimension of $\spe{\mulam}$ equals the number of standard $\mulam$-tableaux. %Unlike in the case of symmetric groups, no closed formula for this number is known. 
Modules of this kind will come back in \cref{seccentalg} and will be the heart of our treatment. We remark that analogous modules for the affine Hecke algebra were studied by Ram \cite{ram}.

From now on throughout this paper we will abuse notation by omitting square brackets and not distinguishing between partitions and Young diagrams and between \skewps and \skewds.

\subsection{The abacus}\label{secabac}
In this section we introduce an idea of James and Kerber \cite{jameskerber} that gives another combinatorial perspective to the objects mentioned in \cref{secpart,secskew}. Fix $e$ a positive integer. We consider an abacus display with $e$ vertical runners labelled from left to right by the symbols $0,\dots,e-1$. We index the \emph{positions} in the $i$-th runner by the integers $i,i+e,i+2e,\dots$ from the top down. Then we align runners so that position $x$ is immediately to the right of position $x-1$ whenever $e\nmid x$. Each position in the abacus is either a \emph{bead} $\abacus(b)$ or a \emph{space} $\abacus(n)$. 

Now consider a partition $\nu$ and an integer $C\ge l(\nu)$. We put a bead in the abacus display just constructed at position $\nu_j-j+C$ for $1\le j\le C$. We refer to this drawing as the \emph{\abac} of $\nu$ (of charge $C$) and we denote it by $\mathrm{Ab}_e(\nu)$, or simply $\Ab{\nu}$ if $e$ is clear from the context. The \abac of a \skewp $\mulam$ is the pair of \abacs $\Ab{\mu}$ and $\Ab{\la}$ with the same charge. 
For convenience we always choose $C\equiv 0\ppmod e$ (this means that the total number of beads in an abacus display is a multiple of $e$). Observe that increasing $C$ by $e$ entails adding a row of beads at the top of the \abac. 
Note that, by construction, every runner of the abacus has infinitely many consecutive spaces downwards. We will use the convention that all positions below those shown are spaces. The picture below exhibits the \abac of $(7,3,2^2,1^2)\vdash16$ in an abacus display with 4 runners and 12 beads.
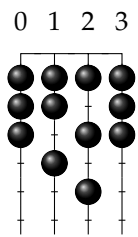
\[
\abacus(e0e1e2e3,lmmr,bbbb,bbnb,bnbb,nbnn,nnbn,nnnn)
\]

We remark that given the \abac $\Ab{\nu}$ with charge $C$, we can read off the $j$-th part of $\nu$ counting the number of spaces which precede the $j$-th largest bead position in $\Ab{\nu}$. Note that, by definition, the abacus encodes the rim of $\nu$: a node $(r,c)$ corresponds to the position marked by $c-r+C$; this position is a bead if and only if $(r,c)$ lies below the next node in the rim of $\nu$, or $(r,c)=(1,\lambda_1)$, and otherwise this position is a space.

We remark that for $h\ge1$, a removable $h$-hook arises from a position $x$ such that $x$ is a bead and $x-h$ is a space in $\Ab\nu$. The removal of this $h$-hook from $\nu$ corresponds to moving the bead at $x$ into the space at $x-h$. Similarly, adding an addable $h$-hook corresponds to moving a bead at position $x$ to position $x+h$, for some $x$. We will say that a removable or addable hook is \emph{at runner $i$} if it corresponds to a position $x\equiv i\ppmod e$. In the \abac above, for example, $(7,3,2^2,1^2)$ has a removable $7$-hook at runner $1$ since position $13$ is a bead and position $6$ is a space. Switching these two positions gives the \abac of the partition $(7,1^2)\vdash9$ obtained by removing nodes $(6,1)$,$(5,1)$,$(4,1)$,$(4,2)$,$(3,2)$,$(2,2)$ and $(2,3)$ from $(7,3,2^2,1^2)$.

%We conclude this section making an observation about \skewps and the abacus. Given $\mulam$ a \skewp, we represent its \abac drawing both the \abacs $\Ab{\mu}$ and $\Ab{\la}$ one above the other. From the last paragraph and \cref{secpart}, $\Ab{\mu}$ and $\Ab{\la}$ only differ in a even number of positions $h_1,\dots,h_k,f_1,\dots,f_k$ such that the $h_i$'s are beads in $\Ab{\mu}$ and spaces in $\Ab{\la}$ and vice versa for the $f_i$'s.  For example, if $\mu=(6,4^2,3,2)\vdash19$ and $\la=(3,2^3,1)\vdash10$, the two corresponding \abacs differ in positions $h_1=5$, $h_2=1$, $h_3=-3$, $f_1=0$, $f_2=-2$, $f_3=-4$ (these are marked by white beads and crosses in the picture below).
%\[
%\mathllap{\Ab{\mu}\qquad}\abacus(bbbb,bbbn,xoxb,xobn,nonn,nnnn)
%\]
%\[
%\mathllap{\Ab{\mu}\qquad}\abacus(bbbb,bbbn,oxob,oxbn,nxnn,nnnn)
%\] %centre the abaci
% \mattnote{Is this observation worthwhile? We don't really use it. We mainly do star beads instead.} \putinote{You're right, we say this later. silenced.}

\subsection{$p$-cores and $p$-blocks}\label{seccore}

Our main interest in this paper is understanding blocks, and to prepare the ground we recall the combinatorics underlying the $p$-block theory of the symmetric groups.  The theory of $p$-blocks of finite groups goes back to the early work of Brauer, and the $p$-blocks for the symmetric group were first studied by Nakayama \cite{naka2}, who conjectured a classification proved simultaneously by Brauer and Robinson \cite{brauer,gdbr}.

Fix an integer $e\ge2$, we say that a partition is an \emph{$e$-core} if it has no removable $e$-hooks. Given $\nu\in\calp$, we define the \emph{$e$-core of $\nu$} as the partition obtained by successively deleting removable $e$-hooks from $\nu$ until an $e$-core is reached. We denote this partition by $\mathrm{core}_e(\nu)$. In terms of the abacus, evaluating $\mathrm{core}_e(\nu)$ coincides with sliding all beads up as far as possible in the runners of $\Ab{\nu}$. Since this process clearly produces a well-defined \abac, we deduce that the $e$-core of a partition is well-defined. We define the \emph{$e$-weight} of $\nu$, written $\mathrm{weight}_e(\nu)$, to be the number of $e$-hooks removed in evaluating $\mathrm{core}_e(\nu)$. Note that $|\nu|=|\mathrm{core}_e(\nu)|+e\,\mathrm{weight}_e(\nu)$. For $i=0,\dots,e-1$, we write $\nu^{(i)}$ for the partition which corresponds to the $i$-th runner of $\Ab{\nu}$ considered as an \abac with one single runner (with positions properly numbered). The $e$-tuple of partitions $(\nu^{(0)},\dots,\nu^{(e-1)})$ is called the \emph{$e$-quotient} of $\nu$. From the abacus interpretation of the weight, we immediately deduce that $\mathrm{weight}_e(\nu)=|\nu^{(0)}|+\dots+|\nu^{(e-1)}|$.

\begin{eg}
Let $\nu=(12,7^2,5,4,2,1^2)\vdash39$ and $e=4$.   Then $\mathrm{core}_4(\nu)=(2,1)\vdash3$, as we see from the following \abacs.
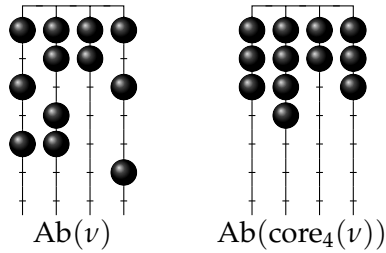
\[
\begin{array}{ccccc}
\abacus(lmmr,bbbb,nbbn,bnnb,nbnn,bbnn,nnnb,nnnn) & & & \abacus(lmmr,bbbb,bbbb,bbnb,nbnn,nnnn,nnnn,nnnn)\\[
32pt]
\Ab{\nu} & & & \Ab{\mathrm{core}_4(\nu)}
\end{array}
\]
We find that $\mathrm{weight}_4(\nu)=9$, and the $4$-quotient of $\nu$ is $((2,1),(1^2),\vn,(3,1))$.
\end{eg}

%Althoug in our treatment the integer $e$ will always be a prime number (since it will act as the characteristic of a finite field), the previous definition, as those in the past sections, are voluntarily given without this assumption. Indeed, for example, these concepts appear and play a relevant role in the representation theory of \emph{$q$-Schur algebras} and \emph{$q$-cyclotomic Hecke algebras} whose \emph{quantum characteristic} may not be a prime (referenza Mathas book?).

Now let $n$ be a positive integer, and consider the $p$-blocks of $S_n$; that is, the indecomposable summands of the algebra $kS_n$. It is a standard fact that each Specht module is a module for a single $p$-block. Conversely, every block contains at least one Specht module (because every simple module is a composition factor of a Specht module), so in order to describe the $p$-block structure of $S_n$ it suffices to say when two Specht modules lie the same block. 
%tools introduced in the previous paragraph are extraordinarily useful in studying the $p$-modular representation of $RS_n$, i.e.\ the structure of the algebra $kS_n$. Since it is not easy to understand the whole group algebra $kS_n$ at one time, the first step in the study of its structure is the reaserch of a decomposition of $kS_n$ into the direct sum of minimal two-sided ideals. These algebras (that are \emph{not} subalgebras of $kS_n$) take the name of \emph{$p$-blocks of $S_n$}. Clearly, the full understanding of every $p$-block of $kS_n$ implies the full understanding of $kS_n$ itself. The set of Specht modules relies a fundamental tool in detecting the $p$-blocks of $S_n$. Actually, they provide a family of \emph{cell modules} for $kS_n$ regarded as a cellular algebra (Mathas?), and so every Specht $kS_n$-module $\spe{\nu}$ is a non trivial $B$-module for a unique $p$-block $B$ of $S_n$ (referenza). In this case, we say that $\spe{\nu}$ \emph{belongs to} $B$. Moreover, all simple $B$-modules are found among the composition factors of the Specht $kS_n$-modules which belongs to $B$ (regarded as $B$-modules). It follows that the problem of decomposing $kS_n$ into $p$-blocks reduces to find a proper subdivision of the Specht $kS_n$-modules and hence of $\calp_n$. 
%
% We are in the exact moment when $p$-cores start playing their role. 
The following famous result, known as the \emph{Nakayama conjecture}, gives a combinatorial condition to understand when two Specht modules belong to the same $p$-block of $S_n$.
\begin{thm}\label{nakconj}
Let $\alpha,\beta\in\calp_n$. The Specht modules $\spe{\alpha},\spe{\beta}$ belong to the same $p$-block of $S_n$ if and only if $\core{\alpha}=\core{\beta}$.
\end{thm}

\cref{nakconj} was stated by Nakayama in 1940 and proved independently by Brauer and Robinson. %Several proofs are known, see for example \cite[Section 6.1]{jameskerber}. Before going on, we dwell a moment on the algebraic impact of \cref{nakconj}. By Maschke's Theorem \cite[Theorem 1.9]{isaacs}, the algebra $kS_n$ is semisimple if and only if $p>n$. If this is the case, Wedderburn's Theorem \cite[Theorem 1.15]{isaacs} tells that the $p$-blocks of $S_n$ are matrix algebras and are in 1:1 correspondence with the set of Specht modules. This means that non-isomorphic Specht modules belong to different $p$-blocks. This fact can be deduced also from \cref{nakconj} noting that, if $p>n$, every partition of $n$ is a $p$-core. If $p\le n$, $kS_n$ is not semisimple and it comes about that some Specht modules belong to the same $p$-block. The Nakayama conjecture shows how to fastly recognise whenever this happens. 
Thorough the paper we will abuse notation regarding a $p$-block $B$ of $S_n$ as the subset of partitions of $n$ which labels the Specht $kS_n$-modules belonging to $B$.

Clearly if two partitions of $n$ have the same $p$-core then they have the same $p$-weight. Given a $p$-block $B$ of $S_n$, we denote by $\gamma_B$ and $w_B$, respectively, the $p$-core and the $p$-weight of \emph{any} partition in $B$. It is clear that (the \abac of) all partitions in $B$ can be recovered starting from $\gamma_B$ and then sliding $w_B$ times beads down in $\Ab{\gamma_B}$ in all possible ways. 

We now give some more conditions equivalent to those in \cref{nakconj}. Note that nothing of what follows requires $p$ to be a prime number. We need all the \abacs to contain the same number of beads. Since the length of a partition of $n$ is at least $n$, to ensure this is enough to fix the (common) charge $C$ of the abaci to be at least $n$, maintaining the convention that $C\equiv0\ppmod p$. For every $\nu\in\calp_n$ and every $i=0,\dots,p-1$, let $b_i(\nu)$ denote the number of beads in the $i$-th runner of the \abac $\mathrm{Ab}_p(\nu)$. It turns out that the $p$-tuple $(b_0(\nu),\dots,b_{p-1}(\nu))$ is another invariant of the $p$-block to which $\nu$ belongs. Actually, by definition, the $p$-core of a partition is obtained by only sliding beads up in its \abac. It follows that the number of beads in every runner remains unchanged at the end of the process. In particular, if $\alpha$ and $\beta$ are in the same $p$-block $B$, then $b_i(\alpha)=b_i(\gamma_B)=b_i(\beta)$ for all $i=0,\dots,p-1$. (Note that the numbers $b_i$ depend on the choice of the charge; in fact $C$ can be recovered as $b_0+\dots+b_{p-1}$. If $C$ is increased by $p$, then each $b_i$ is increased by $1$.)

Analysing the integers just defined, Ellers and Murray \cite[Lemma 7]{elmu} derived the following result which we will need later.

\begin{propn}\label{prop6em}
Let $B$ be a $p$-block of $S_n$ and let $\alpha,\beta\in B$. Suppose that for some $h\ge1$ and some $i\in\{0,\dots,p-1\}$, $\alpha$ has a removable $h$-hook at runner $i$. Then at least one of the following occurs:
\begin{itemize}
\item $\alpha$ has an addable $h$-hook at runner $j\equiv i-h\ppmod p$;
\item $\beta$ has a removable $h$-hook at runner $i$.
\end{itemize}
\end{propn}

Another reinterpretation of the combinatorial condition for two Specht modules to be in the same block concerns $p$-contents. Recall that the $p$-content of a partition $\la$ is the multiset of $p$-residues of the nodes of $\la$. The following result goes back to Littlewood \cite{litt}.

\begin{thm}\label{littthm}
Suppose $\la,\mu\in\calp_n$. Then $\la$ and $\mu$ have the same $p$-core \iff they have the same $p$-content.
\end{thm}

So we can label a $p$-block $B$ of $S_n$ by a multiset of cardinality $n$ of elements of $\zpz$, corresponding to the $p$-content of any partition in $B$. Using this interpretation, we immediately gain a useful fact which will come in handy in \cref{secbelt}.

\begin{propn}\label{beltcore}
Let $m$ be an integer with $m\ge p$. If $\mulam\in\calp_{m-p,m}$ and $\resp{\mulam}=\{\widebar0,\dots,\widebar{p-1}\}$, then $\core{\la}=\core{\mu}$ and $\mathrm{weight}_p(\la)=\mathrm{weight}_p(\mu)-1$.
\end{propn}

\begin{proof}
Let $\nu=\core\la$, let $B$ be the block of $S_{m-p}$ containing $\la$, and let $C$ be the block of $S_m$ with $p$-core $\nu$. Observe that the nodes of a rim $p$-hook have contents $a,a+1,\dots,a+p-1$ for some integer $a$, and therefore have $p$-residues $\widebar0,\widebar1,\dots,\widebar{p-1}$ in some order. Therefore the $p$-content of $C$ is obtained from the $p$-content of $B$ by adding one copy of each element $\widebar0,\widebar1,\dots,\widebar{p-1}$. From \cref{littthm}, this means that $\mu$ lies in $C$.
\end{proof}

\subsection{The degenerate affine Hecke algebra}\label{secdaha}

The next two sections are devoted to introducing the main algebraic objects in the paper. Before meeting the algebra which is at the heart of the treatment, we introduce an auxiliary structure and some interesting related concepts. Let $A$ be a field and $n$ be a positive integer. 
\begin{defn}\label{defdaha}
The \emph{degenerate affine Hecke algebra of degree $n$ over $A$}, denoted by $\calh_n^A$, is the unital, associative $A$-algebra with generators $z_1,\dots,z_n,s_1,\dots,s_{n-1}$ subject to the following relations:
\begin{enumerate}
\item $z_iz_j=z_jz_i$ for all $i,j$;
\item $(s_is_j)^{m_{ij}}=1$, where $m_{ii}=1$, $m_{i(i+1)}=3$ and $m_{ij}=2$ for $|i-j|>1$;
\item $s_iz_j=z_js_i$, for $j\neq i,i+1$;
\item $s_iz_i=z_{i+1}s_i-1$.
\end{enumerate}
\end{defn}
Looking at the relation $(1)$, we see that the \emph{polynomial} generators $z_1,\dots,z_n$ generate a subalgebra $A[z_1,\dots,z_n]$ of $\calh_n^A$ isomorphic to the polynomial algebra in $n$ indeterminates over $A$. Moreover, this is a maximal commutative subalgebra of $\calh_n^A$. Considering relations in $(2)$, we also observe that the \emph{Coxeter} generators $s_1,\dots,s_{n-1}$ generate a subalgebra of $\calh_n^A$ isomorphic to $AS_n$. In fact, as $A$-modules, $\calh_n^A\cong A[z_1,\dots,z_n]\otimes_A AS_n$. We remark that the centre $Z(\calh_n^A)$ of $\calh_n^A$ is given by the ring of symmetric polynomials in the indeterminates $z_1,\dots,z_n$ \cite[Theorem 3.3.1]{klbook}. 

We now focus on the polynomial subalgebra in the case where $A$ is an algebraically closed field. Since $A[z_1,\dots,z_n]$ is commutative, every simple $A[z_1,\dots,z_n]$-module $M$ is one-dimensional and corresponds (via a 1:1 correspondence) to the $n$-tuple $(a_1,\dots,a_n)\in A^n$ such that $z_i$ acts as the scalar $a_i$ on $M$. We define the \emph{\forch} of a $A[z_1,\dots,z_n]$-module as the formal $\mathbb{Z}$-linear combination of the elements of $A^n$ corresponding to its composition factors. The \emph{\forch} of a finite dimensional $\calh_n^A$-module $M$ is then defined as the \forch of its restriction $M\downarrow_{A[z_1,\dots,z_n]}$. More formally, following \cite{klbook}, one identifies $A[z_1,\dots,z_n]$ with the \emph{parabolic subalgebra}
\[
\calh_{(1,\dots,1)}^A\cong\calh_1^A\otimes_A\dots\otimes_A\calh_1^A
\]
of $\calh_n^A$, and defines the \forch $\ch{M}$ of $M$ as the image of the class $[M]$ in the Groethendieck group $K(\calh_n^A\textrm{-mod})$, under the homomorphism
\[
\mathrm{ch}:K(\calh_n^A\textrm{-mod})\longrightarrow K(\calh_{(1,\dots,1)}^A\textrm{-mod})
\]
induced by the restriction from $\calh_n^A$ to $\calh_{(1,\dots,1)}^A$. By \cite[Theorem 5.3.1]{klbook}, $\mathrm{ch}$ is injective, therefore the composition factors of an $\calh_n^A$-module are determined by its \forch. This fact will be particularly useful in what follows due to the strong relationship between the degenerate affine Hecke algebra and the algebra we are going to meet in the next section.

Before going on we associate another important object to any indecomposable $\calh_n^A$-module $M$. Let $(a_1,\dots,a_n)$ be a summand in the \forch of $M$. The \emph{\cench} of $M$ is defined as the map $Z(\calh_n^A)\rightarrow A$ such that $f\mapsto f(a_1,\dots,a_n)$ for all $f\in Z(\calh_n^A)$. Since $M$ is indecomposable, all summands of $\ch{M}$ lie in the same $S_n$-orbit \cite[Lemma 4.2.2]{klbook}, so the previous definition does not depend on the choice of the summand $(a_1,\dots,a_n)$. So we say that $\{a_1,\dots,a_n\}$ is the \cench of $M$ and we write $\cch{M}$ for it. Then two indecomposable $\calh_n^A$-modules lie in the same block \iff they have the same \cench.

\subsection{The centraliser algebra for symmetric groups}\label{seccentalg}
We are ready to meet the algebraic protagonist of the present work. Most of the material in this section is taken from \cite{elmu}. Let $A$ be a commutative ring and take $m,l$ non-negative integers with $l\le m$ and set $n\vcentcolon=m-l$. Then $AS_l$ is an $A$-subalgebra of the group ring $AS_m$. We define the \emph{centraliser algebra} $\cenalgA$ to be the centraliser of $AS_l$ in $AS_m$; that is, the algebra
\[
\cenalgA\vcentcolon=\{a\in AS_m\ |\ ab=ba\ \text{for every}\ b\in AS_l\}.
\]
This algebra has been studied extensively by Ellers and Murray \cite{elmu1,elmu}, following earlier more general work on centraliser algebras for subgroups of finite groups by Ellers. However, $\cenalgA$ remains poorly understood, and in particular its blocks are not yet known. This is the focus of the present paper.

By the very effective description in \cite[Section 2.1]{klbook} we know a straightforward generating set for $\cenalgA$. This is the union of three subsets:
\begin{itemize}
\item the centre $Z(AS_m)$;
\item the symmetric group (isomorphic to $S_n$) on the symbols $\{l+1,\dots,m\}$;
\item the set of \emph{Jucys--Murphy elements} $L_j=(1\ j)+(2\ j)+\dots+(j-1\ j)$ for $j=l+1,\dots,m$.
\end{itemize}

%More generally, if $G$ is a finite group and $H\le G$, the algebra $AG^H$, i.e.\ the centraliser of $AH$ in $AG$, is isomorphic, in the notation of \cite[Section 11D]{currei}, to the Hecke algebra $\calh(H\times G,\ H\times H,\ \mathbbm{1}_{H\times H})$ over $A$, where $\mathbbm{1}_{H\times H}$ denotes the trivial $A(H\times H)$-module. By this isomorphism and the theory in \cite{currei} it follows that if both algebras $AG$ and $AH$ are semisimple, then so is $AG^H$. \textcolor{red}{Question: This paragraph here or in the introduction?}\mattnote{I think the introduction, but I don't feel strongly.}\putinote{Let's see when we write the introduction.}

We introduce a family of $\cenalgA$-modules which will be at the heart of our treatment. Suppose that $\la\vdash l$ and $\mu\vdash m$. Then the $A$-module $\spe{\mulam}=\Hom_{AS_l}(\spe{\la},\spe{\mu}\downarrow_{S_l})$ introduced at the end of \cref{secskew} is naturally a module for $\cenalgA$, via the action
\[
(c\phi)(v)=c\phi(v)\qquad\text{for }c\in\cenalgA,\ \phi\in\spe{\mulam},\ v\in\spe{\la}.
\]
As already remarked, $\spe{\mulam}$ is non-trivial if and only if $\mulam\in\calp_{l,m}$. To emphasise the similarity with the symmetric group case, we refer to these $\cenalgA$-modules as \emph{Specht modules}. 

We now recall our $p$-modular system $(R,\bbf,k)$, where $p$ is the prime characteristic of the residue field $k$. Since $\mathrm{char}(\bbf)=0$ the algebras $\bbf S_m$ and $\bbf S_l$ are both semisimple by Maschke's Theorem. The general theory of centraliser algebras then implies that $\cenalgF$ is semisimple as well, with a complete irredundant set of simple modules given by the set
\[
\{\spe{\mulam}\ |\ \mulam\in\calp_{l,m}\}
\]
of Specht modules.

Things change if the fraction field $\bbf$ is replaced by the residue field $k$: the centraliser algebra $\cenalgk$ is in general not semisimple and its Specht modules generally fail to be simple. The decomposition number problem for $\cenalgk$ asks for the composition factors of the Specht modules; unfortunately, at present we do not have a good labelling of simple module for $\cenalgk$. This paper is concerned with determining the \emph{$p$-blocks} of $\cenalgk$. The Specht module $\spe\mulam$ over $k$ is a $p$-modular reduction of the corresponding (simple) Specht module over $\bbf$. So (analogously to the case of symmetric groups) describing the $p$-block structure of $\cenalgk$ reduces to finding the appropriate partition of the set of Specht $\cenalgk$-modules. For this, we use an approach pioneered by Ellers and Murray, which we describe in the next section.

\subsection{Combinatorial blocks}\label{seccombbl}

Let $l,m,n$ be as in \cref{seccentalg}. We have pointed out that in order to find the decomposition of $\cenalgk$ into $p$-blocks, we need to look for a way to partition the set of Specht $\cenalgk$-modules. We know by \cref{nakconj} how the Specht modules for the symmetric group split into $p$-blocks, this is a characterisation based on the concept of $p$-core. Our main conjecture (which is implicit in the work of Ellers and Murray \cite{elmu}) predicts that the case of the centraliser algebra works in the same way.

\begin{conj}\label{mainconj}
Let $\mu_1\setminus\la_1,\mu_2\setminus\la_2\in\calp_{l,m}$. The Specht $\cenalgk$-modules $\spe{\mu_1\setminus\la_1},\spe{\mu_2\setminus\la_2}$ belong to the same $p$-block of $\cenalgk$ if and only if $\core{\mu_1}=\core{\mu_2}$ and $\core{\la_1}=\core{\la_2}$.
\end{conj}

We can consider \cref{mainconj} in terms of idempotents. We say that a non-zero central idempotent of $\cenalgk$ is a \emph{block idempotent} if it is primitive, i.e.\ it cannot be written as the sum of two non-zero central idempotents. The blocks of $\cenalgk$ are then precisely the algebras $e\cenalgk$ for $e$ a primitive central idempotent.

In the case of the symmetric group $S_n$, \cref{nakconj} tells that there is a bijection between the set of $p$-block idempotents of $kS_n$ and the set of $p$-core partitions which are the $p$-core of at least one partition in $\calp_n$. Given $p$-blocks $B$ and $B'$ of $kS_m$ and $kS_l$ respectively, let $e_B$ and $e_{B'}$ the corresponding $p$-block idempotents. Then it is straightforward to see that $\ebeb$ is a central idempotent of $\cenalgk$, so is a sum of block idempotents, and $\ebeb\cenalgk$ is a (possibly zero) sum of $p$-blocks of $\cenalgk$. Moreover, $\ebeb$ acts as the identity on $\spe{\mu\sm\la}$ if $\spe\mu$ lies in $B$ and $\spe\la$ lies in $B'$ (that is, if $\core\mu=\ga_B$ and $\core\la=\ga_{B'}$), and otherwise $\ebeb\spe{\mu\sm\la}=0$. If $\ebeb\neq0$, then we refer to $\ebeb\cenalgk$ as a \emph{\conjbl} of $\cenalgk$. \cref{mainconj} then predicts that every combinatorial block is a $p$-block. Since the sum of the elements $\ebeb$ over all choices of $B$ and $B'$ is $1$, this would then mean that every $p$-block of $\cenalgk$ is a combinatorial block. We will abuse notation by identifying a combinatorial block with the set of Specht modules in that combinatorial block, or with the set of labelling \skewps.

%\begin{conj}\label{conjidem}
%  Let $e_B$ and $e_{B'}$ be block idempotents of $kS_m$ and $kS_l$ respectively. If $\ebeb\ne0$, then $\ebeb$ is a block idempotent of $\cenalg$. 
%  \end{conj}

In \cite{elmu}, Ellers and Murray prove \cref{mainconj} in the case $n\ls3$. The aim of the present paper is to prove \cref{mainconj} for a large family of \conjbls of $\cenalgk$. We follow the idea of Ellers and Murray \cite[Section 1.2]{elmu} to exploit the relationship between the centraliser algebra $\cenalgk$ and the degenerate affine Hecke algebra $\calh_n^k$. Their work shows that every $\cenalgk$-module naturally admits the structure of an $\calh_n^k$-module: the symmetric group on $\{l+1,\dots,m\}$ and the Jucys--Murphy elements $L_{l+1},\dots,L_m$ generate a subalgebra of $\cenalgk$ which is a quotient of $\calh_n^k$; restricting to this subalgebra and then inflating yields an $\calh_n^k$-module. With respect to this structure, every Specht module is equipped with a \forch and a \cench and it is an easy task to combinatorially evaluate these. Crucially, two $\cenalgk$-modules in the same combinatorial block have the same composition factors \iff the corresponding $\calh_n^k$-modules do. This enables us to work with formal characters of Specht modules to prove the $p$-block structure.

Let $\mulam\in\calp_{l,m}$ and consider a standard $\mulam$-tableau $T$. %As remarked in \cref{secskew}, there is a path in the branching graph of symmetric groups between the Specht modules $\spe{\la}$ and $\spe{\mu}$ which corresponds to $T$. This can be thought as a sequence of nodes $\mathbf{n}_0,\dots,\mathbf{n}_{n-1}$ such that $\la\cup\{\mathbf{n}_0\}$ is a partition of $l+1$, $\la\cup\{\mathbf{n}_0,\mathbf{n}_1\}$ is a partition of $l+2$ and so on. 
We associate to $T$ the permutation $a_T=(a_1,\dots,a_n)$ of $\resp{\mulam}$ such that $a_j$ is the $p$-residue of the node containing with number $j$ in $T$, for $j=1,\dots,n$. The \forch of the Specht module $\spe{\mulam}$, we denote by $\ch{\mulam}$, is the formal sum of the $n$-tuples $a_T$ with $T$ varying in the set of standard $\mulam$-tableaux. 

\begin{eg}
Let $p=3$ and $\mulam=\sk{(4^2,2)}{(2^2,1)}\in\calp_{5,10}$. The \forch of the Specht $\calc_{5,10}^k$-module $\spe{\mulam}$ is as follows (we draw the \skewd $\mulam$ filling every node with its 3-residue, omitting bars):
\[
\begin{tikzpicture}[scale=1]
\tyng(0cm,0cm,4,4,2)
\Ylinethick{1.5pt}
\tgyoung(0cm,0cm,::;2;0,::;1;2,:;2)
\end{tikzpicture}
\]
\[
2\cdot(2,2,0,1,2)+2\cdot(2,2,1,0,2)+2\cdot(2,0,1,2,2)+2\cdot(2,1,0,2,2)+(2,0,2,1,2)+(2,1,2,0,2).
\]
\end{eg}

By construction, every term of the \forch $\ch{\mulam}$ is a permutation of the $p$-content of $\mulam$. The \cench of the Specht module $\spe{\mulam}$, as an $\calh_n^k$-module, then exactly coincides with the multiset $\resp{\mulam}$. We observe the following fact.

\begin{propn}
Let $\mu_1\setminus\la_1$ and $\mu_2\setminus\la_2$ be \skewps in the same \conjbl of $\cenalgk$. Then $\resp{\mu_1\setminus\la_1}=\resp{\mu_2\setminus\la_2}$.
\begin{pf}
This follows from \cref{littthm}.
\end{pf}
\end{propn}

Motivated by the above result, we say that a \conjbl $B$ has \cench $\cch{B}$ equal to the $p$-content of \emph{any} \skewp in $B$.

One of the main tools in the paper is the \emph{decomposition matrix}. Let $B$ be a \conjbl and let $D_B$ be the matrix whose rows are indexed by $B$ and columns by those simple $\daha$-modules appearing as composition factors of some Specht module in $B$. To detect the entries of the row of $D_B$ labelled by the \skewp $\mulam$, write $\ch{\mulam}$ as a $\mathbb{Z}$-linear combination of \forchs of simple $\calh_n^k$-modules. The coefficients in this expression give the desired decomposition numbers. $B$ is then a $p$-block of $\cenalgk$ if and only if the decomposition matrix $D_B$ is a connected matrix (i.e.\ if it cannot be put in block-diagonal form by permuting rows and columns).

We end this section and this chapter with a crucial remark. 

\begin{rmk}
Given a \skewp $\mu\sm\la$, fill the boxes of the skew-diagram $\mu\sm\la$ with their $p$-residues, and define the \emph{\conf} of $\mu\sm\la$ to be the multiset of \conncomps of the resulting diagram.  It is easy to see that if two \skewps in $B$ have the same \conf, then the corresponding Specht modules have the same formal character, and therefore the same composition factors (with multiplicity) as $\calh_n^k$-modules. This means that for our purposes, we can largely shorten the number of rows of $D_B$ relabelling them with the set $\calb$ that comes out quotienting $B$ by the equivalence relation that identifies \skewps with the same \conf. We say that $\calb$ is the \emph{set of \confs} of $B$. We regard the elements of $\calb$ as $\daha$-modules and when talking about the \forch and the composition factors of a \conf $X$ we mean those of \emph{any} Specht module in $B$ whose \conf is $X$. Also, saying that a \conf belongs to a \conjbl we mean that the \conjbl contains a \skewp with that \conf.

% \mattnote{It's not true that two Specht modules with the same \conf are isomorphic, even if they're in the same combinatorial block. e.g.\ take $(3,1)\sm(2)$ and $(2,1^2)\sm(1^2)$ in characteristic $2$.}\putinote{wow. Naively I though yes.}
\end{rmk}

When drawing \confs, we follow \cite{elmu} by drawing the \conncomps in a diagonal line in arbitrary order.

For example, take $p=3$, and consider the \skewps $(4,1)\sm(2)$ and $(3,2)\sm(2)$:
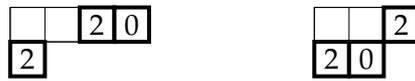
\[
\begin{tikzpicture}[scale=1]
\tyng(0cm,0cm,4,1)
\tyng(4cm,0cm,3,2)
\Ylinethick{1.5pt}
\tgyoung(0cm,0cm,::;20,2)
\tgyoung(4cm,0cm,::;2,20)
\end{tikzpicture}
\]
Both \skewps have the same $3$-shape and the corresponding Specht modules both have formal character $(2,0,2)+2\cdot(2,2,0)$.

\section{Ribbon blocks and belt blocks}\label{rbsec}

This chapter is devoted to establishing \cref{mainconj} for a special family of \conjbls of the centraliser algebra. Again let $m,l$ be non-negative integers with $m\ge l$ and let $n=m-l$. Recall that $(R,\bbf,k)$ is a $p$-modular system where $p$ denotes the (prime) characteristic of the residue field $k$. Let $B$ be a \conjbl of $\cenalgk$. The \cench of $B$ is a multiset of elements of $\zpz$. We will be concerned with the case where there are no repeated entries in the \cench, so that the \cench is a subset of $\zpz$. In this case clearly $n\ls p$ and we say that $B$ is a \emph{\grb} if $n<p$ (so that the \cench is a proper subset of $\zpz$), or a \emph{\belbl} if $n=p$. We will assume for the rest of \cref{rbsec} that $B$ is either a \grb or a \belbl.

In the last section we explained that when working modulo $p$ we can identify \skewps with the same \forch (since the rows of the decomposition matrix $D_B$ labelled by them are identical), so from now on we will be concerned with just the set $\calb$ of \confs of $B$. Firstly it is a good idea to give a visualisation of the elements of $\calb$. If $\mulam$ is a \skewp in $B$, then its \skewd does not contain a $2\times2$-subdiagram (because otherwise it would have boxes with equal $p$-residue, contrarily to our assumption on $B$) and hence it is a disjoint union of rim-hooks of $\mu$. If $X$ is the \conf of $\mulam$, we refer to a \emph{\conncomp} of $X$ for any of these rim-hooks and we say that $X$ is a \emph{\rib} if $\mulam$ is connected, i.e.\ it is a single rim-hook of $\mu$. %We draw a \conf forgetting the partitions from which it comes from and drawing the underlined \skewd fullfilling every node with the residue modulo $p$ of its content. %Our convention in drawing a disconnected \conf is to order its \conncomps from NE to SW as the residues of their hand nodes, decreasingly. 
For example, if $p>n=3$, and $B$ has \cench $\{\widebar0,\widebar1,\widebar2\}$, then the \confs in $\calb$ are among the nine drawings below (the first four on the left being the possible \ribs).

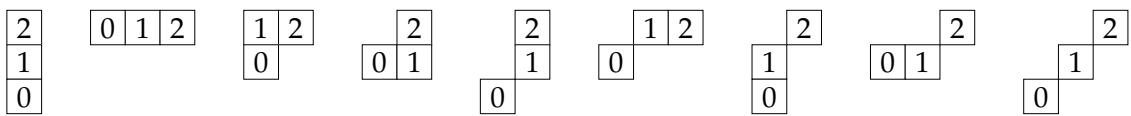
\newcommand\sepp{\hspace*{18pt}}

\[
\begin{array}{c@{\sepp}c@{\sepp}c@{\sepp}c@{\sepp}c@{\sepp}c@{\sepp}c@{\sepp}c@{\sepp}c}
\young(2,1,0)
&
\young(012)
&
\young(12,0)
&
\gyoung(:;2,01)
&
\gyoung(:;2,:;1,0)
&
\gyoung(:;12,0)
&
\gyoung(:;2,1,0)
&
\gyoung(::;2,01)
&
\gyoung(::;2,:;1,0)
\end{array}
\]

We close this long preface noting that, since there are no repetitions in $\cch{B}$, if $X$ is a \conf in $\calb$, there is no risk of confusion writing $\bxi$ for the unique node labelled by $i$ in $X$, for every $i\in\cch{B}$.

\subsection{\Grbs}\label{secrib}

In this section we fix a \grb $B$, with $\calc=\calc(B)\subset\zpz$ its \cench. Our aim is to prove \cref{mainconj} for $B$.

Before starting we give another combinatorial perspective to the \confs in a \grb. The reason why we do this will be clarified in the next section. We define an \emph{\arrg} to be a directed graph with vertex set $\calc$ such that every arrow has the form $i\rightarrow i+1$ or $i+1\rightarrow i$. Given a \conf $X$, we define an associated \arrg $\arrX$ by drawing an arrow $i\rightarrow i+1$ [resp.\ $i+1\rightarrow i$] whenever $\bxi$ lies on the left of [resp.\ below] $\bxipo$ in $X$, or no arrow between $i$ and $i+1$ if these nodes lie in different \conncomps of $X$. It is clear that the set of all possible \confs is in bijection with the set of all possible \arrgs. We define a \conf to be \emph{maximal} if its \arrg is maximal, i.e.\ contains an arrow between $i$ and $i+1$ whenever $i,i+1\in\calc$. Note that starting from an \arrg we can recover the corresponding \conf by mirroring the above procedure. Moreover, we can evaluate its \forch by taking the formal sum of all permutations $(a_1,\dots,a_n)$ of $\calc$ with the property that $i$ appears before $j$ for every arrow $i\rightarrow j$.
\begin{eg}
Suppose $p=7$ and $\calc=\{\widebar0,\widebar1,\widebar2,\widebar3,\widebar6\}$. We give an example of a \conf and the corresponding \arrg (omitting bars).
\Yvcentermath1
\[
\gyoung(::;3,:;12,0,6)
\qquad\qquad
\begin{tikzpicture}[scale=1.2]
\draw(0,0)node(0){$0$};
\draw(1,0)node(1){$1$};
\draw(2,0)node(2){$2$};
\draw(3,0)node(3){$3$};
\draw(-1,0)node(6){$6$};
\draw[->](1)--(2);
\draw[->](3)--(2);
\draw[->](0)--(6);
\end{tikzpicture}
\]
The corresponding \forch is then
\[
\begin{array}{r}
(0,1,3,2,6)+(0,1,3,6,2)+(0,1,6,3,2)+(0,3,1,2,6)+(0,3,1,6,2)\\+\,(0,3,6,1,2)+(0,6,1,3,2)+(0,6,3,1,2)+(1,0,3,2,6)+(1,0,3,6,2)\\+\,(1,0,6,3,2)+(1,3,0,6,2)+(1,3,0,2,6)+(1,3,2,0,6)+(3,0,1,2,6)\\+\,(3,0,1,6,2)+(3,0,6,1,2)+(3,1,0,6,2)+(3,1,0,2,6)+(3,1,2,0,6)\mathrlap.
\end{array}
\]
\end{eg}

\begin{rmk}
	Let $\mulam$ be a \skewp with \conf $X$. Consider the abacus displays for $\mu$ and $\la$ with the same number of beads. For each $i\in\calc$, let $x_i$ be the position of $\Ab\mu$ corresponding to the node $\bxi$ (that is a node of the rim of $\mu$ by assumption). We look at the arrow graph of $X$ to detect the content of the position $x_i$ in $\Ab\mu$ and $\Ab\la$: $x_i$ is a space in both $\Ab\mu$ and $\Ab\la$ if there is an arrow $i\rightarrow i+1$, is a bead in both $\Ab\mu$ and $\Ab\la$ if there is an arrow $i+1\rightarrow i$, is a bead in $\Ab\mu$ and a space in $\Ab\la$ if there is no arrow between $i$ and $i+1$.
\end{rmk}

Now we introduce a fundamental tool in our treatment. Let $X,Y$ be two \confs and let $\Dist{X}{Y}$ the subset of $\calc$ containing those integers $i$ such that $i+1\in\calc$ and $\bxi$ lies on the left of $\bxipo$ in $X$ and below $\bxipo$ in $Y$ (see the picture below), or vice versa.
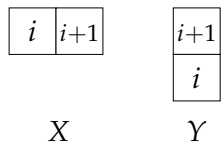
\[
\begin{array}{c@\qquad c}
\youngs(\wb,i\ipo)
&
\youngs(\wb,\ipo,i)
\\[24pt]
X
&
Y
\end{array}
\]

Define the \emph{distance $\dist{X}{Y}$ between $X$ and $Y$} as the cardinality of $\Dist{X}{Y}$. This concept can be rephrased also in terms of \arrgs: take $\arrX$ and $\arrY$ and denote by $\arrX\cup\arrY$ the picture obtained by overlapping these two \arrgs. Then $\dist{X}{Y}$ coincides with the number of edges where two differently oriented arrows overlap in $\arrX\cup\arrY$.

Note that $\operatorname d$ is \emph{not} a metric: $\dist XY$ can be zero even when $X\neq Y$.

\begin{eg}
Suppose $p\gs11$ and $\calc=\{\widebar0,\widebar1,\widebar2,\widebar3,\widebar4,\widebar5,\widebar6,\widebar7,\widebar8,\widebar9\}$ and consider the two following \confs:
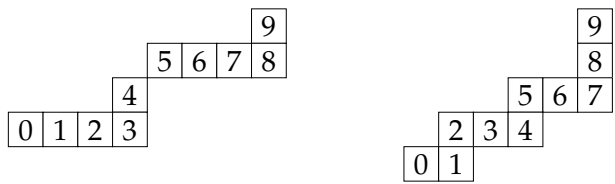
\[
\gyoung(:::::::;9,::::;5678,:::;4,0123)\qquad\qquad
\gyoung(:::::;9,:::::;8,:::;567,:;234,01)
\]
\Yvcentermath0
Then $\Dist{X}{Y}=\{\widebar1,\widebar3,\widebar7\}$ and $\dist{X}{Y}=3$. The union of the two arrow graphs $\arrX\cup\arrY$ is given as follows, with $\arrX$ in red and $\arrY$ in blue.

\[
\begin{tikzpicture}[scale=1.1]
\draw(0,0)node(0){$0$};
\draw(1,0)node(1){$1$};
\draw(2,0)node(2){$2$};
\draw(3,0)node(3){$3$};
\draw(4,0)node(4){$4$};
\draw(5,0)node(5){$5$};
\draw(6,0)node(6){$6$};
\draw(7,0)node(7){$7$};
\draw(8,0)node(8){$8$};
\draw(9,0)node(9){$9$};
\ruar{red}01
\rdar{\myblue}01
\ruar{red}23
\rdar{\myblue}23
\ldar{\myblue}54
\ruar{red}56
\rdar{\myblue}56
\ruar{red}67
\rdar{\myblue}67
\luar{red}98
\ldar{\myblue}98
\ruar{red}12
\ldar{\myblue}21
\luar{red}43
\rdar{\myblue}34
\ruar{red}78
\ldar{\myblue}87
\end{tikzpicture}
\]

\end{eg}

We start our analysis of \grbs by observing that the \forchs of two different \gribs have no summands in common.

\begin{lemma}\label{disjch}
Consider $\bfa=(a_1,\dots,a_n)$ where $a_1,\dots,a_n$ are the elements of $\calc$ in some order. Then $\bfa$ appears in the \forch of exactly one \grib.

\begin{proof}
We construct the unique \grib $X$ such that $\bfa$ is a term of the \forch $\ch{X}$. Start drawing $\bxi$, for some $i\in\calc$ such that $i-1\notin\calc$. Then consider the order of $i$ and $i+1$ in $\bfa$: if $i+1$ occurs before [resp.\ after] $i$, then draw $\bxipo$ just above [resp.\ to the right of] the node $\bxi$. Now do the same for $i+1$ and $i+2$, and continue until an element $k\in\calc$ is reached for which $k+1\notin\calc$. This determines one component of $X$. Now repeat with the other elements $i\in\calc$ for which $i-i\notin\calc$ to obtain the other components.

By construction, $\bfa$ is a summand of $\ch{X}$. We now see that $X$ is unique. If $Y$ is different \grib, then the set $\Dist{X}{Y}$ is non-empty, because $\arrX$ and $\arrY$ both have arrows joining $i$ and $i+1$ whenever $i,i+1\in\calc$. Take $i\in\Dist{X}{Y}$, and suppose that $i$ appears before $i+1$ in $\bfa$ (the other case being similar). This means that we have an arrow $i\to i+1$ in $\arrX$, and therefore an arrow $i\leftarrow i+1$ in $\arrY$, so that $\bfa$ does not appear in $\ch Y$.
\end{proof}
\end{lemma}

\begin{cory}\label{noincommon}
Different \gribs have no composition factors in common.
\end{cory}

The next three results will finally give us a labelling set for the columns of the decomposition matrix of a \ribbl. In \cref{refs}, if $X$ is a \conf, a standard $X$-tableau indicates a standard tableau of \emph{any} \skewp with \conf $X$. We begin with a definition.

\begin{defn}
Let $X$ be a \conf. A \emph{\reft} of $X$ is a \grib $Y$ that can be obtained joining the \conncomps of $X$ together.
\end{defn}

\begin{eg}
Suppose $p=7$, and let $X$ be the \conf from an earlier example:
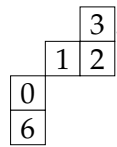
\[
\gyoung(::;3,:;12,0,6).
\]
Then $X$ has exactly two refinements, as follows.
\[
\gyoung(:;3,12,0,6)\qquad\gyoung(::;3,012,6).
\]
\end{eg}

\begin{lemma}{\label{refs}}
Let $X$ be a \conf. The \forch of $X$ is the sum of the \forchs of the \refts of $X$.
\begin{proof}
This follows from a correspondence between standard tableaux: given a \reft $Y$ of $X$, each standard $Y$-tableau gives rise to a standard $X$-tableau (just by breaking into individual \conncomps). Conversely, each standard $X$-tableau arises in this way for a unique \reft $Y$: for each $i$ such that the nodes $\bxi$ and $\bxipo$ lie in different \conncomps of $X$, the relative order of the entries $i,i+1$ tells us how they need to be joined together if the resulting tableau is to be standard.
\end{proof}
\end{lemma}

\begin{cory}{\label{compfact}}
If $X$ is a \conf, then the composition factors of $X$ are the composition factors of the \refts of $X$.
\begin{proof}
This follows from the fact that the \forchs of simple modules of the degenerate affine Hecke algebra are linearly independent \cite[Theorem 5.3.1]{klbook}.
\end{proof}
\end{cory}

\begin{thm}{\label{ribequiv}}
Let $X$ and $Y$ be two \confs. The following are equivalent:
\begin{enumerate}
\item $X$ and $Y$ have a composition factor in common;
\item The \forchs of $X$ and $Y$ have a term in common;
\item $\dist{X}{Y}=0$;
\item There is a \grib which is a \reft of both $X$ and $Y$.
\end{enumerate}
\begin{proof}
\begin{description}[beginthm]
\item[\rm (1$\Rightarrow2$)] This follows from the fact that the \forch of a $\calh_n^k$-module is the sum of the \forchs of its composition factors.
\item[\rm ($2\Rightarrow3$)] If there exists an integer $i\in D(X,Y)$, then $i+1$ occurs before $i$ in every summand of $\ch{X}$ and after $i$ in every summand of $\ch{Y}$, or vice versa. So $\ch{X}$ and $\ch{Y}$ have no terms in common.
\item[\rm ($3\Rightarrow4$)] Consider the graph $\arrX\cup\arrY$. Since $\dist{X}{Y}=0$, this is an admissible \arrg. Then every \arrg obtained by filling the vacant edges in $\arrX\cup\arrY$ corresponds to a \reft of both $X$ and~$Y$.
\item[\rm ($4\Rightarrow1$)] This follows from \cref{compfact}.\qedhere
\end{description}
\end{proof}
\end{thm}

\cref{ribequiv} tells us that the set of \gribs serves as a labelling set for an approximation to the decomposition matrix of $B$: by \cref{compfact} any simple module $S$ occurs as a composition factor of a unique \grib $Y$ (say with multiplicity $c$). So by \cref{compfact}, if $X$ is any \conf then $S$ occurs as a composition factor of $X$ with multiplicity $c$ if $Y$ is a refinement of $X$, and $0$ otherwise.

So we can define a matrix $D_B$ with rows indexed by the set of \confs, and columns indexed by the set of \gribs, with entry $d_{XY}=1$ if $Y$ is a refinement of $X$ and $0$ otherwise. Then the actual decomposition matrix of $B$ is obtained from $D_B$ by possibly duplicating rows (because a given \conf may correspond to more than one Specht module) and possibly duplicating and rescaling columns (because a Specht module labelled by a \grib might be reducible). As a consequence, to show that $B$ is a single $p$-block of $\cenalgk$, we just need to show that $D_B$ is a connected matrix, and our problem becomes purely combinatorial.

(In fact we suspect that every Specht module labelled by a \grib is simple, but we could not find a proof of this statement, and we do not need it.)

\Yvcentermath1
\begin{eg}
Suppose $p=5$, $l=5$ and $m=8$, and let $B$ be the block consisting of Specht modules $\spe{\mu\sm\la}$ where $\corep5\mu=(2,1)$ and $\corep5\la=\vn$. Then $\calc=\{\widebar0,\widebar1,\widebar4\}$, and the \confs in $B$ are as follows:
\[
X_1=\young(401),\quad X_2=\gyoung(:;01,4),\quad X_3=\young(01,4),\quad X_4=\gyoung(:;1,0,4),\quad X_5=\young(1,0,4).
\]
(For example, the \conf $X_2$ arises from the \skewps $(7,1)\sm(5)$ and $(5,3)\sm(4,1)$.)

The matrix $D_B$ is then given as follows.
\[
\begin{array}{c|ccc}
&X_1&X_3&X_5
\\\hline
X_1&1&0&0
\\
X_2&1&1&0
\\
X_3&0&1&0
\\
X_4&0&1&1
\\
X_5&0&0&1
\end{array}
\]
\end{eg}

\begin{defn}
Let $X,Y$ be two \confs in $\calb$. We say that $X$ and $Y$ are \emph{$p$-linked} if there exists a series of \confs $X=X_1,X_2,\dots,X_{t-1},X_t=Y$ in $\calb$ such that $\dist{X_j}{X_{j+1}}=0$ for every $1\le j\le t-1$.
\end{defn}

In what follows we will mainly work with \abacs. We remark that if $\mulam$ is a \skewp in $B$, a \conncomp of its $p$-shape corresponds to a rim-hook of $\mu$. If $h,f\in\mathbb{Z}$ are the positions corresponding to the hand node and foot node of the rim-hook, respectively, then $h$ is a bead and $f-1$ is a space in $\Ab{\mu}$, and vice versa in $\Ab{\lambda}$. These pairs of positions are the only in which the contents in $\Ab{\mu}$ and $\Ab\la$ differ. We represent $\Ab{\mu}$ and $\Ab{\la}$ simultaneously in an \abac $\Ab{\mulam}$, putting a \emph{right half-bead} $\abacus(M)$ at position $h$ and a \emph{left half-bead} $\abacus(L)$ at position $f-1$.  (So the right half-beads are precisely the beads occurring in the \abac of $\mu$ but not of $\la$, and vice versa for the left half-beads.)

\begin{eg}
Take $p=5$ and $\mu\sm\la=(7,3,1^2)\sm(7,2)$, giving the following \conf. \[\gyoung(:;1,3,2)\] This \conf has two connected components, so that $\mu\sm\la$ comprises two rim-hooks. We draw the \abac $\Ab{\mulam}$ with half-beads.
\[\
	\abacus(lmmmr,bbbbb,bLbMn,LMnnn,nbnnn)
	\]
\end{eg}

Our aim is to show that any two \confs in $B$ are $p$-linked. Consider a pair of \confs $X,Y\in\calb$ such that $\dist XY>0$. Let $\X$ and $\Y$ be \skewps in $B$ having \conf $X$ and $Y$, respectively. We write $\abx$ and $\aby$ for the \abds $\Ab\X$ and $\Ab\Y$, respectively. Consider $\calc$ the \cench of the generalised ribbon block $\calb$; recall that this is a (disjoint) union of intervals of $\zpz$. For every $j\in\calc$, consider the (unique) node $\bxj$ in $\X$ and $\Y$ and denote by $x_j$ and $y_j$ the positions corresponding to these nodes in the \abds $\abx$ and $\aby$, respectively. The choice of the charge ensures that positions $x_j$ and $y_j$ are both in the $j$-th runner.

Note that in $\abx$, position $x_j$ is
\begin{itemize}
\item
a bead, if node $\bxjpo$ lies immediately above $\bxj$ in $\X$;
\item
a space, if $\bxjpo$ lies immediately to the right of $\bxj$ in $\X$;
\item
a right half-bead, if $\bxj$ and $\bxjpo$ are in different components of $\X$.
\end{itemize}
In particular, $x_j$ cannot be a left half-bead in $\abx$. A similar statement applies to $y_j$, so that $j\in\Dist XY$ \iff $x_j$ is a bead in $\abx$ and $y_j$ is a space in $\aby$, or vice versa.

Fix $i\in\Dist{X}{Y}$, and assume \wolog that the node $\bxi$ lies on the left of $\bxipo$ in $X$ and below it in $Y$. Then $x_i$ is a space in $\abx$ and $y_i$ is a bead in $\aby$. We let $r$ be maximal such that $i,i+1,\dots,i+r\in\calc$ (where we add modulo $p$); note that there must be such an $r$ because by assumption $\calc\neq\zpz$. Observe that if $j\in\{i,\dots,i+r\}$ and node $\bxj$ lies in the same connected component of $\X$ as $\bxi$, then $x_i,x_{i+1},\dots,x_j$ are consecutive integers in increasing order, so that $x_j=x_i+j-i$; a similar statement holds for $y_j$.

We prove a series of useful facts.

% \begin{lemma}\label{lemma1}
% 	Let $k\in\{i+1,\dots,i+r\}$ be the minimum such that $k\notin\Dist XY$ and:
% 	\begin{itemize}
% 		\item $x_k$ is a bead or a right half-bead in $\abx$ and $\mu_X$ has an addable $(k-i)$-hook at runner $i$, or
% 		\item $y_k$ is a space or a right half-bead in $\aby$ and $\la_Y$ has a removable $(k-i)$-hook at runner~$k$.
% 	\end{itemize}
% 	Then position $y_j+i-j$ is a bead in $\aby$ and position $x_j+i-j$ is a space in $\abx$, for every $j\in\{i+1,\dots,k\}$.
% \end{lemma}

\begin{lemma}\label{lemma1}
Suppose $k\in\{i+1,\dots,i+r\}$, and suppose that for every $j\in\{i+1,\dots,k-1\}\sm\Dist XY$:
\begin{itemize}
\item
$x_j$ is a space in $\abx$ or $\mu_X$ has no addable $(j-i)$-hook at runner $i$, and
\item
$y_j$ is a bead in $\aby$ or $\la_Y$ has no removable $(j-i)$-hook at runner~$j$.
\end{itemize}
Then position $y_j+i-j$ is a bead in $\aby$ and position $x_j+i-j$ is a space in $\abx$, for every $j\in\{i+1,\dots,k\}$.
\end{lemma}

% Suppose k \in {i+1,...,i+r}, and suppose that for every l in {i+1,...,k-1}, either l \in D(X,Y), or:
% x_l is a space or left half-bead in X or mu_X has no addable (l-i)-hook at runner i,
% and
% y_l is a bead or left half-bead in Y or la_Y has no removable (l-i)-hook at runner l.

\begin{pf}
We only prove the first part of the statement; the second follows by a perfectly dual argument.

Consider the node $\bxj$ in $Y$. Let $t\in\bbz$ be such that $y_j=y_i-i+j+tp$. If $t=0$, then position $y_j+i-j$ is a bead since $y_i$ is by assumption. If $t\ne0$, then (from the observation just before the \lcnamecref{lemma1}) $\bxi$ and $\bxj$ lie in different \conncomps of $Y$. Let $l\in\{i+2,\dots,j\}$ be such that $\bxl$ is the foot node of the \conncomp of $Y$ containing $\bxj$. This implies that $y_l=y_j+l-j$, so that $y_j+i-j=y_l+i-l$. Moreover, $\bxlmo$ is the hand node of a \conncomp of $Y$, so that $y_{l-1}$ is a right half-bead in $\aby$ and $l-1\notin\Dist XY$. The hypotheses of the \lcnamecref{lemma1} then imply that $\la_Y$ has no removable $(l-1-i)$-hooks at runner $l-1$, so position $y_l+i-l$ must be a bead in $\aby$ because $y_l-1$ is a left half-bead.  %The discussion preceding the statement ensures that this is a bead position in $\Ab{\mu_Y}$, as well.
\end{pf}

% Before proceding we observe that a $k$ as in \cref{lemma1} exists. Consider $i+r$. The node $\bxipr$ is the hand node of a \conncomp of both $X$ and $Y$ because $i+r+1\notin\calc$. Then positions $x_{i+r}$ and $y_{i+r}$ are right half-beads in $\abx$ and $\aby$, respectively. In particular, $i+r\notin\Dist XY$. Moreover, $\mu_X$ has a removable $r$-hook at runner $i+r$ since position $x_{i+r}$ is a right half-bead and $x_{i+r}-r$ is a space by \cref{lemma1}. \putinote{Is it okay to refer to it?} By \cref{prop6em} it follows that $\mu_X$ has an addable $r$-hook at runner $i$ or $\mu_Y$ has a removable $r$-hook at runner $i+r$. In the latter case, the removable hook does not occur at position $y_{i+r}$ because position $y_{i+r}-r$ is a bead by \cref{lemma1}, then it is a removable $r$-hook at runner $i+r$ of $\la_Y$ as well. Hence, $i+r$ satisfies at least one of the conditions in the definition of $k$.

\begin{lemma}\label{lemma2}
	Let $k$ be as in \cref{lemma1} and let $j\in\{i+1,\dots,k-1\}\sm\Dist XY$.
	\begin{enumerate}
	\item If $x_j$ is a space in $\abx$, then $\mu_X$ has an addable $(j-i)$-hook at runner $i$;
	 \item If $y_j$ is a bead in $\aby$, then $\la_Y$ has a removable $(j-i)$-hook at runner $j$.
	 \end{enumerate}
\end{lemma}

\begin{pf}	
We only show (1); a dual argument then applies to (2). 
The hypothesis that $j\notin\Dist XY$ implies that $y_j$ is a space or a right half-bead in $\aby$. On the one hand, the hypothesis on $k$ means that $\la_Y$ has no removable $(j-i)$-hooks at runner $j$. On the other hand, by \cref{lemma1}, position $y_j+i-j$ is a bead in $\aby$ meaning that $\la_Y$ has an addable $(j-i)$-hook at runner $i$. By \cref{prop6em}, it follows that $\la_X$ has an addable $(j-i)$-hook at runner $i$. This hook is not at position $x_i$ because this is a space by assumption, hence it is an addable $(j-i)$-hook at runner $i$ of $\mu_X$ as well.
\end{pf}
	
We now give the main result of this section.

\begin{thm}\label{grbthm}
Let $X,Y\in\calb$ such that $\dist XY>0$. Then there is $Z\in\calb$ such that $\dist{X}{Z}<\dist XY$ and $\dist{Y}{Z}<\dist XY$.
\end{thm}

\begin{pf}
Let all the setting be as above and let $k\in\{i+1,\dots,i+r\}$ be maximal subject to the hypotheses of \cref{lemma1}. Then we claim that $k\notin\Dist XY$ and either
\begin{itemize}
\item
$x_k$ is a bead or a right half-bead in $\Ab{\X}$ and $\mu_X$ has an addable $(k-i)$-hook at runner $i$
\end{itemize}
or
\begin{itemize}
\item
$y_k$ is a space or a right half-bead in $\Ab{\Y}$ and $\la_Y$ has a removable $(k-i)$-hook at runner~$k$.
\end{itemize}
If $k<i+r$ then this follows from the assumption that $k$ is maximal subject to the hypotheses of \cref{lemma1}, so suppose instead that $k=i+r$. The node $\bxipr$ is the hand node of a \conncomp of both $X$ and $Y$ because $i+r+1\notin\calc$. Then positions $x_{i+r}$ and $y_{i+r}$ are right half-beads in $\abx$ and $\aby$, respectively, and in particular $i+r\notin\Dist XY$. Moreover, $\mu_X$ has a removable $r$-hook at runner $i+r$ since position $x_{i+r}$ is a right half-bead and $x_{i+r}-r$ is a space by \cref{lemma1}. By \cref{prop6em} it follows that $\mu_X$ has an addable $r$-hook at runner $i$ or $\mu_Y$ has a removable $r$-hook at runner $i+r$. In the latter case, the removable hook does not occur at position $y_{i+r}$ because position $y_{i+r}-r$ is a bead by \cref{lemma1}, then it is a removable $r$-hook at runner $i+r$ of $\la_Y$ as well. Hence, $i+r$ satisfies at least one of the above conditions.

We show that the required \conf exists by modifying $\abx$ or $\aby$, according to which of the above conditions $k$ satisfies. We show how to proceed in the first case, with the other case following by analogous and dual arguments. So we assume that $x_k$ is a bead or a right half-bead in $\abx$ and $\mu_X$ has an addable $(k-i)$-hook at runner $i$. 

The strategy is to operate on the abacus display $\abx$ by performing a series of bead moves as shown in \cref{mainthmfig}. Note that the rows of the abacus in which the given positions appear are illustrative only; these rows could easily appear in a different order in practice. The abacus is arranged such that $i$ labels the leftmost runner and $k$ the rightmost one.
\begin{figure}[hbt]
\[
\begin{tikzpicture}[label distance=-4.5pt]
	\foreach\g in{4.5}{
		\foreach\x in {0,1.5,3,4.5}\draw(\x,-.5)--++(0,8)++(\g+4.5,0)--++(0,-8);
		%beads
		\draw(0,0)node[label=below left:${a_{s-1}=a_2}$](as-1){\ };\draw(as-1)node{$\abacus(b)$};\draw[red](as-1)node{$\abacus(L)$};
		\draw(0,1)node[label=below left:$a_3$](a3){\ };\draw(a3)node{$\abacus(b)$};
		\draw(0,2)node[label=below left:$a_1$](a1){\ };\draw(a1)node{$\abacus(b)$};
		\draw(0,3)node[label=below left:${a_s=c_s}$](as){\ };\draw(as)node{$\abacus(b)$};\begin{scope}[scale=.6]\bluertbll05\end{scope}
		\draw(3,0)node[label=below left:$c_{s-1}$](cs-1){\ };\draw(cs-1)node{$\abacus(b)$};\begin{scope}[scale=.6]\bluertbll{5}0\end{scope}
		\draw(4.5+\g,0)node[label=below left:$c_2$](c2){\ };\draw(c2)node{$\abacus(b)$};\begin{scope}[scale=.6]\bluertbll{7.5+5*\g/3}0\end{scope}
		\draw(7.5+\g,2)node[label=below left:$c_1$](c1){\ };\draw(c1)node{$\abacus(b)$};\begin{scope}[scale=.6]\bluertbll{12.5+5*\g/3}{10/3}\end{scope}
		\draw(9+\g,7)node[label=below right:$x_k$](xk){\ };\draw(xk)node{$\abacus(b)$};\begin{scope}[scale=.6]\redrtbll{15+5*\g/3}{35/3}\end{scope}
		\draw(1.5,3)node(xjd){\ };\draw(xjd)node{$\abacus(b)$};
		\draw(4.5,0)node(xje){\ };\draw(xje)node{$\abacus(b)$};
		\draw(6+\g,0)node(xjf){\ };\draw(xjf)node{$\abacus(b)$};
		%spaces
		\draw(0,7)node(xja){\ };\draw(xja)node{$\abacus(n)$};
		\draw(0,6)node(xjb){\ };\draw(xjb)node{$\abacus(n)$};
		\draw(0,4)node(xjc){\ };\draw(xjc)node{$\abacus(n)$};
		\draw(0,5)node[label=below left:$x_i$](xi){\ };\draw(xi)node{$\abacus(n)$};
		\draw(3,6)node(xjh){\ };\draw(xjh)node{$\abacus(n)$};
		\draw(3,3)node[label=below right:$d_s$](ds){\ };\draw(ds)node{$\abacus(n)$};
		\draw(4.5+\g,1)node[label=below right:$d_3$](d3){\ };\draw(d3)node{$\abacus(n)$};
		\draw(4.5+\g,5)node(xjj){\ };\draw(xjj)node{$\abacus(n)$};
		\draw(7.5+\g,7)node(xjs){\ };\draw(xjs)node{$\abacus(n)$};
		\draw(7.5+\g,0)node[label=below right:$d_2$](d2){\ };\draw(d2)node{$\abacus(n)$};
		\draw(9+\g,2)node[label=below right:$d_1$](d1){\ };\draw(d1)node{$\abacus(n)$};
		%right half-beads
		\draw(1.5,5)node[label=below right:$x_{j_1}$](xj1){\ };%\draw(xj1)node{$\abacus(M)$};
		\begin{scope}[scale=.6]\redrtbll{2.5}{25/3}\end{scope}
		\draw(4.5,6)node[label=below right:$x_{j_2}$](xj2){\ };%\draw(xj2)node{$\abacus(M)$};
		\begin{scope}[scale=.6]\redrtbll{7.5}{10}\end{scope}
		\draw(6+\g,5)node[label=below right:$x_{j_t}$](xjt){\ };%\draw(xjt)node{$\abacus(M)$};
		\begin{scope}[scale=.6]\redrtbll{10+5*\g/3}{25/3}\end{scope}
		%left half-beads
		\draw(1.5,6)node(xjp){\ };\draw(xjp)node{$\abacus(L)$};
		\draw(4.5,4)node(xjq){\ };\draw(xjq)node{$\abacus(L)$};
		\draw(6+\g,7)node(xjr){\ };\draw(xjr)node{$\abacus(L)$};
		%dots
		\draw(.75,5)node{\dots};
		\draw(.75,3)node{\dots};
		\draw(2.25,6)node{\dots};
		\draw(2.25,3)node{\dots};
		\draw(3.75,6)node{\dots};
		\draw(3.75,0)node{\dots};
		\draw(5.25,0)node{\dots};
		\draw(5.25,4)node{\dots};
		\draw(3.75+\g,1)node{\dots};
		\draw(3.75+\g,5)node{\dots};
		\draw(5.25+\g,5)node{\dots};
		\draw(5.25+\g,0)node{\dots};
		\draw(6.75+\g,0)node{\dots};
		\draw(6.75+\g,7)node{\dots};
		\draw(8.25+\g,7)node{\dots};
		\draw(8.25+\g,2)node{\dots};
		%arrows
		\draw(xjt)++(-2,.6)coordinate(xx);
		\draw(cs-1)++(2,.6)coordinate(cc);
		\draw[->,thick,red](xj1)to[out=150,in=30](xi);
		\draw[->,thick,red](xj2)to[out=150,in=30](xjp);
		\draw[->,thick,red](xk)to[out=150,in=30](xjr);
		\draw[->,thick,red](xjq)++(1,.3)to[out=180,in=30](xjq);
		\draw[->,thick,red](xjt)to[out=150,in=0](xx);
		\draw[->,thick,\myblue](as)to[out=30,in=150](ds);
		\draw[->,thick,\myblue](c2)to[out=30,in=150](d2);
		\draw[->,thick,\myblue](c1)to[out=30,in=150](d1);
		\draw[->,thick,\myblue](d3)+(-1,.3)to[out=0,in=150](d3);
		\draw[->,thick,\myblue](cs-1)to[out=30,in=180](cc);
	}
\end{tikzpicture}
\]
\caption{$\Ab{X}$.}\label{mainthmfig}
\end{figure}

In terms of the Young diagram, what we are going to do is to remove all the nodes $\bxipo,\dots,\young(k)$ from $\mu_X$, and to put these nodes back in new positions.

We work on $\abx$. By assumption there exists $a_1\equiv i\ppmod p$ such that $\mu_X$ has an addable $(k-i)$-hook at a bead position $a_1$. 
Let $1\le k_1\le k-i$ be the least integer such that $k-k_1\notin\Dist{X}{Y}$, position $x_{k-k_1}$ is a space and $a_1+k-i-k_1$ is a bead. By \cref{lemma2}, if $k_1<k-i$, then there exists an integer $a_2\equiv i\ppmod p$ such that $\mu_X$ has an addable $(k-i-k_1)$-hook at position $a_2$. Then we repeat the above process, letting $1\ls k_2\ls k-i-k_1$ be minimal such that $k-k_1-k_2\notin\Dist XY$, $x_{k-k_1-k_2}$ is a space and $a_2+k-i-k_1-k_2$ is a bead. It is clear that this procedure always terminates at a step $s$ such that $i=k-\sum_{1\le u\le s}k_u$. Summarising, we have found integers $a_1,\dots,a_s\equiv i\ppmod p$ and $k_1,\dots,k_s$ such that, for every $1\le v\le s$, in $\abx$, there is a bead at position $c_v\vcentcolon=a_v+k-i-\sum_{1\le u\le v}k_u$ and a space at position $d_v\vcentcolon=a_v+k-i-\sum_{1\le u\le v-1}k_u$. Let $i+1\le j_1<\dots<j_t\le k$ be those residues such that $\young(\jo),\dots,\young(\jt)$ are hand nodes of \conncomps of $X$.

We now construct a new \abac by modifying $\abx$. Our aim is to modify $\mu_X$ while keeping $\la_X$ the same, so each move only moves the right half of a bead. We perform the following moves (where the colours correspond to the coloured arrows in \cref{mainthmfig}):
\begin{itemize}
\item \textcolor{red}{move the right half-bead (or the right half of the bead) at $x_k$ to the left half-bead at $x_{j_t+1}-1$},
\item \textcolor{red}{for every $2\le v\le t$, move the right half-bead at $x_{j_v}$ to the left half-bead at $x_{j_{v-1}+1}-1$},
\item \textcolor{red}{move the right half-bead at $x_{j_1}$ to the space at $x_i$},
\item \textcolor{\myblue}{for every $1\le v\le s$, move the right half of the bead at $c_v$ to the space at $d_v$}.
\end{itemize}

If $\mu_Z$ is the partition of $m$ defined by this new \abac, then by construction, $\Xmod$ is a \skewp in the \conjbl $B$ (we have replaced each removed node with a node with the same residue). Let $Z$ be the \conf of this \skewp. We will show that in most cases, $Z$ is the \conf desired in the \lcnamecref{grbthm}. For every $j\in\calc$, let $z_j$ be the position of the node $\bxj$ in $\Ab{\Xmod}$. Note that for $j\notin\{i+1,\dots,k\}$, then $z_j=x_j$. Also, $\{z_{i+1},\dots,z_k\}$ is the union of intervals $[c_1,d_1],\dots,[c_s-1,d_s]$, e.g. $z_k=d_1$.% \putinote{Is this note useful?}yes!

We firstly note that $i\notin\Dist XZ\cup\Dist YZ$ because $z_i$ is a right half-bead in $\Ab{\Xmod}$ by construction. If $j\in\{i+1,\dots,k-1\}$ with $j\notin\Dist XY$, then $j\notin\Dist XZ$. Indeed, if $x_j$ is a bead $\abx$, then by the definition of $k$,  $\mu_X$ has no addable $(j-i)$-hooks at runner $i$, so that $z_j$ is a bead in $\Ab{\Xmod}$; if $x_j$ is a space in $\abx$, then the choice of $k_1,\dots,k_s$ means that $z_j$ is a space or a right half-bead in $\Ab{\Xmod}$. Similarly it can be deduced that $j\notin\Dist YZ$. %(note that the inductive hypothesis means that it cannot be the case that $x_j$ is a star bead and $y_j$ is a space). Why did we write this?

Furthermore, $k\notin\Dist XZ$ because $x_k$ and $z_k$ are beads in the respective \abacs. It follows that $\dist XZ<\dist XY$. Regarding to $Y$, it can happen that $k\in D(Y,Z)$: this situation occurs \iff $y_k$ is a space in $\aby$, $x_k$ is right half-bead in $\abx$ and $a_1=x_{k+1}-k+i-1$. This condition can be rephrased in terms of Young diagrams, saying that $\bxk$ is on the left of $\bxkpo$ in $Y$, and during the process from $\mu_X$ to $\mu_Z$, the node $\bxk$ -- that is the hand node of a \conncomp of $X$ -- has joined below the node $\bxkpo$. If this is note the case, then $\dist YZ<\dist XY$ and $Z$ is the required \conf. 
%\mattnote{This needs a bit of work: the point is that $D(X,Z)\subseteq D(X,Y)\sm\{i\}$. In particular, given a $j\notin D(X,Y)$, the above deductions about $x_j$ and $z_j$ mean straight away that $j\notin D(X,Z)$. We can also deduce that $j\notin D(Y,Z)$, but this is fiddlier: it might appear possible to have $j\in D(Y,Z)$ if $y_j$ is a space, $x_j$ is a star bead and $z_j$ is a (non-star)bead. But if $j<k$ then the inductive hypothesis implies in particular that it can't be the case that $x_j$ is a star bead and $y_j$ a space. This leaves only the case $j=k$, which is what the rest of the paragraph deals with.} 

Assume for the rest of the proof that position $y_k$ is a space in $\aby$, position $x_k$ is a right half-bead in $\abx$ and $a_1=x_{k+1}-k+i-1$. We can also assume that the only addable $(k-i)$-hook at runner $i$ of $\mu_X$ is at position $a_1$. Indeed, if $\mu_X$ had another such an addable hook, then the above process of modification of $\mu_X$ could be initiated from this addable hook, resulting a different partition $\mu_Z\vdash m$ that works for our purposes. By \cref{lemma1}, position $x_k+i-k$ is a space in $\abx$, therefore $\mu_X$ has a removable $(k-i)$-hook at runner $k$ (at position $x_k$). Again by \cref{lemma1}, position $y_k+i-k$ is a bead in $\aby$, therefore $\mu_Y$ has an addable $(k-i)$-hook at runner $i$ (at position $y_k+i-k$). The fact that $\mu_X$ has a removable $(k-i)$-hook at runner $k$ and exactly one addable $(k-i)$-hook at runner $i$, together with \cref{prop6em}, implies that $\mu_Y$ has a removable $(k-i)$-hook at runner $k$. This is not at position $y_k$ (because $y_k$ is a space), then it is a removable hook of  $\la_Y$ as well. Now we can perform a dual version of the bead moves above operating on $\aby$ and constructing another \skewp in $B$. In terms of young diagrams, we add all the nodes $\bxipo,\dots,\young(k)$ to $\la_Y$ and we remove these nodes from other positions. Since the construction is perfectly dual, we only show an illustative drawing (\cref{mainthmfig2}) that gives an idea of the bead moves. (Again, the labels of the shown runners go from $i$ to~$k$.)
\begin{figure}[hbt]
\[
\begin{tikzpicture}[label distance=-4.5pt]
	\foreach\g in {4.5}{
		\foreach\x in {0,1.5,3,4.5}\draw(\x,-.5)--++(0,8)++(\g+4.5,0)--++(0,-8);
		%beads
		\draw(1.5,0)node(b){\ };\draw(b)node{$\abacus(b)$};\begin{scope}[scale=.6]\redlfbll{2.5}0\end{scope}
		\draw(4.5,3)node(d){\ };\draw(d)node{$\abacus(b)$};\begin{scope}[scale=.6]\redlfbll{7.5}5\end{scope}
		\draw(6+\g,1)node(f){\ };\draw(f)node{$\abacus(b)$};\begin{scope}[scale=.6]\redlfbll{10+5*\g/3}{5/3}\end{scope}
		\draw(9+\g,2)node(g){\ };\draw(g)node{$\abacus(b)$};\begin{scope}[scale=.6]\redlfbll{15+5*\g/3}{10/3}\end{scope}
		
		\draw(0,7)node[label=below left:$y_i$](yi){\ };\draw(yi)node{$\abacus(b)$};\begin{scope}[scale=.6]\bluelfbll0{35/3}\end{scope}
		\draw(0,4)node(4){\ };\draw(4)node{$\abacus(b)$};
		\draw(0,5)node(5){\ };\draw(5)node{$\abacus(b)$};
		\draw(0,6)node(6){\ };\draw(6)node{$\abacus(b)$};
		
		\draw(1.5,7)node(7){\ };\draw(7)node{$\abacus(b)$};
		\draw(4.5,5)node(n){\ };\draw(n)node{$\abacus(b)$};
		\draw(6+\g,5)node(8){\ };\draw(8)node{$\abacus(b)$};
		
		%left half-beads
		\draw(3,5)node(s2){\ };%\draw(s2)node{$\abacus(L)$};
		\begin{scope}[scale=.6]\bluelfbll5{25/3}\end{scope}
		\draw(4.5+\g,5)node(s4){\ };%\draw(s4)node{$\abacus(L)$};
		\begin{scope}[scale=.6]\bluelfbll{7.5+5*\g/3}{25/3}\end{scope}
		\draw(7.5+\g,6)node(s5){\ };%\draw(s5)node{$\abacus(L)$};
		\begin{scope}[scale=.6]\bluelfbll{12.5+5*\g/3}{10}\end{scope}
		
		%right half-beads
		\draw(3,7)node(m){\ };\draw(m)node{$\abacus(M)$};
		\draw(4.5+\g,4)node(o){\ };\draw(o)node{$\abacus(M)$};
		\draw(7.5+\g,5)node(p){\ };\draw(p)node{$\abacus(M)$};
		
		%spaces
		\draw(0,0)node(a){\ };\draw(a)node{$\abacus(n)$};
		\draw(1.5,3)node(c){\ };\draw(c)node{$\abacus(n)$};
		\draw(4.5,2)node(e){\ };\draw(e)node{$\abacus(n)$};
		\draw(6+\g,2)node(h){\ };\draw(h)node{$\abacus(n)$};
		\draw(0,1)node(1){\ };\draw(1)node{$\abacus(n)$};
		\draw(0,2)node(2){\ };\draw(2)node{$\abacus(n)$};
		\draw(0,3)node(3){\ };\draw(3)node{$\abacus(n)$};

		\draw(7.5+\g,2)node(r){\ };\draw(r)node{$\abacus(n)$};
		%\draw[label=below right:$y_k$](6+\g,6)node(yk){\ };\draw(yk)node{$\abacus(n)$};
		\draw(9+\g,6)node[label=below right:$y_k$](yk){\ };\draw(yk)node{$\abacus(n)$};
		
		\draw(3,3)node(9){\ };\draw(9)node{$\abacus(n)$};
		\draw(4.5+\g,1)node(10){\ };\draw(10)node{$\abacus(n)$};
		
		%dots
		\draw(.75,0)node{\dots};
		\draw(.75,7)node{\dots};
		\draw(2.25,3)node{\dots};
		\draw(2.25,7)node{\dots};
		\draw(3.75,3)node{\dots};
		\draw(3.75,5)node{\dots};
		\draw(5.25,2)node{\dots};
		\draw(5.25,5)node{\dots};
		
		\draw(3.75+\g,1)node{\dots};
		\draw(3.75+\g,4)node{\dots};
		\draw(5.25+\g,1)node{\dots};
		\draw(5.25+\g,5)node{\dots};
		\draw(6.75+\g,2)node{\dots};
		\draw(6.75+\g,5)node{\dots};
		\draw(8.25+\g,2)node{\dots};
		\draw(8.25+\g,6)node{\dots};
		
		%arrows
		\draw(f)++(-2,.6)coordinate(xx);
		\draw(s2)++(2,.6)coordinate(yy);
		\draw[->,thick,red](b)to[out=150,in=30](a);
		\draw[->,thick,red](d)to[out=150,in=30](c);
		\draw[->,thick,red](e)++(1,.3)to[out=180,in=30](e);
		\draw[->,thick,red](f)to[out=150,in=0](xx);
		\draw[->,thick,red](g)to[out=150,in=30](h);

		\draw[->,thick,\myblue](yi)to[out=30,in=150](m);
		\draw[->,thick,\myblue](o)+(-1,.3)to[out=0,in=150](o);
		\draw[->,thick,\myblue](s2)to[out=30,in=180](yy);
		\draw[->,thick,\myblue](s4)to[out=30,in=150](p);
		\draw[->,thick,\myblue](s5)to[out=30,in=150](yk);
}
\end{tikzpicture}
\]
\caption{$\Ab{Y}$.}\label{mainthmfig2}
\end{figure}

By construction, we obtain a partition $\la_W\vdash l$ such that $\Ymod\in B$. We let $W$ be the \conf of $\Ymod$. Similar to the above with $Z$, it can be shown that $i\notin\Dist XW\cup\Dist YW$ and for every $j\in\{i+1,\dots,k-1\}$ with $j\notin\Dist{X}{Y}$, then $j\notin\Dist XW\cup\Dist YW$. Also, $k\notin\Dist YW$ and hence $\dist YW<\dist XY$.  Finally, by the assumption that $x_k$ is a right half-bead in $\abx$, it follows that $k\notin\Dist XW$. This implies that $\dist XW<\dist XY$ and therefore that $W$ is the required \conf.
\end{pf}

\begin{eg}
Let $m=42$, $l=37$ and $p=7$. The \skewps $\X=(14^2,7,5,2)\setminus(14^2,4,3,2)$ and $\Y=(9,8,7,5^2,4^2)\setminus(9,8,4^5)$ have size $n=m-l=5$ and belong to the same \conjbl of $\calc_{37,42}^{k}$ having $7$-content $\calc=\{\widebar0,\widebar1,\widebar2,\widebar3,\widebar4\}$. Indeed, $\mathrm{core}_7(\mu_X)=\mathrm{core}_7(\mu_Y)=(5,3,2^3)$ and $\mathrm{core}_7(\la_X)=\mathrm{core}_7(\la_Y)=(7,3,2^3)$. Let $X$ and $Y$ be the \confs of $\X$ and $\Y$, respectively. The Young diagrams and the \abacs (with charge 14) are shown below.

\Yvcentermath0
\allowdisplaybreaks
\begin{align*}
&\X&&\begin{tikzpicture}[scale=.6,baseline=1cm]
\draw(0,3.5*\yoy)--++(6,0);
\foreach\x in{0,1,2,3,4,5,6}{\bll\x3}
\bll02\bll12\bll42\spce22\spce32\spce52\lfbll62
\spce01\bll11\spce21\spce31\rtbll41\spce51\spce61
\foreach\x in{0,1,2,3,4}{\spce\x0}
\bll50\bll60
\end{tikzpicture}
&&
\raisebox{20pt}{$\gyoung(;;;;;;;;;;;;;;,;;;;;;;;;;;;;;,;;;;!{\Ylinethick{1.5pt}}234,!{\Ylinethick{.3pt}};;;!{\Ylinethick{1.5pt}}01,!{\Ylinethick{.3pt}};;)$}
\\[20pt]
&\Y&&\begin{tikzpicture}[scale=.6,baseline=1cm]
\draw(0,3.5*\yoy)--++(6,0);
\foreach\x in{0,1,2,3,4,5,6}{\bll\x3}
\foreach\x in{0,1,2,3}{\spce\x2}\bll42\bll52\lfbll62
\bll01\bll11\spce21\spce31\rtbll41\spce51\bll61
\foreach\x in{0,2,3,4,5,6}{\spce\x0}\bll10
\redlfbll10
\redlfbll42
\bluelfbll01
\draw(1-.5*\ball,2*\yoy)node(a){\ };
\draw(4-.5*\ball,2*\yoy)node(b){\ };
\draw[->,red,thick](b)to[out=160,in=20](a);
\draw(0-.5*\ball,\yoy)node(c){\ };
\draw(4-.5*\ball,\yoy)node(d){\ };
\draw[->,\myblue,thick](c)to[out=20,in=160](d);
\draw(0-.5*\ball,0)node(e){\ };
\draw(1-.5*\ball,0)node(f){\ };
\draw[->,red,thick](f)to[out=160,in=20](e);
\end{tikzpicture}
&&
\raisebox{20pt}{$\gyoung(;;;;;;;;;,;;;;;;;;,;;;;!{\Ylinethick{1.5pt}}234,!{\Ylinethick{.3pt}};;;;!{\Ylinethick{1.5pt}}1,!{\Ylinethick{.3pt}};;;;!{\Ylinethick{1.5pt}}0,!{\Ylinethick{.3pt}};;;;,;;;;)$}
\end{align*}
We see that $\Dist{X}{Y}=\{\widebar0\}$, so that $\dist{X}{Y}=1$. In particular, position $x_0=y_0=14$ is a space in $\abx$ and a bead in $\aby$. We have that $k=4$ and the construction in \cref{grbthm} comprises the bead moves shown in the \abac $\aby$ above (we are modifying $\la_Y$ but not $\mu_Y$, so we move the left half of each bead, as indicated by the colours). We find a partition $\la_Z=(8^2,7,5,4^2,1)\vdash37$ such that $\mu_Y\setminus\la_Z$ has $7$-shape $Z$ such that $\dist{X}{Z}=\dist{Y}{Z}=0$ as can be easily deduced from the Young diagrams.
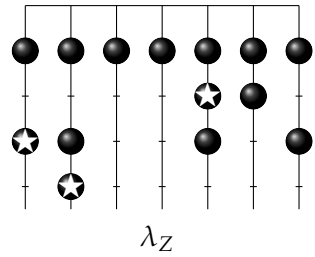
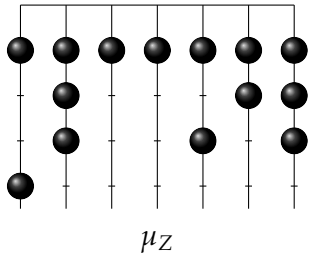
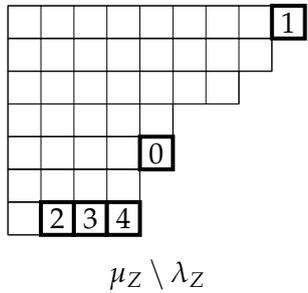
\begin{align*}
&\mu_Y\sm\la_Z&&\begin{tikzpicture}[scale=.6,baseline=1cm]
\draw(0,3.5*\yoy)--++(6,0);
\foreach\x in{0,1,2,3,4,5,6}{\bll\x3}
\foreach\x in{0,2,3}{\spce\x2}\lfbll12\rtbll42\bll52\lfbll62
\rtbll01\bll11\spce21\spce31\bll41\spce51\bll61
\foreach\x in{2,3,4,5,6}{\spce\x0}\lfbll00\rtbll10
\end{tikzpicture}
&&
\raisebox{20pt}{$\gyoung(;;;;;;;;!{\Ylinethick{1.5pt}}1,!{\Ylinethick{.3pt}};;;;;;;;,;;;;;;;,;;;;;,;;;;!{\Ylinethick{1.5pt}}0,!{\Ylinethick{.3pt}};;;;,;!{\Ylinethick{1.5pt}}234!{\Ylinethick{.3pt}})$}
\end{align*}
\end{eg}

We immediately deduce the following \lcnamecref{ribthm}, using induction on $\dist XY$.

\begin{cory}\label{ribthm}
Suppose $X,Y\in\calb$. Then $X$ and $Y$ are $p$-linked.
\end{cory}

We deduce as a corollary the main result of this section.

\begin{cory}\label{ribcor}
\cref{mainconj} holds for \grbs.
\end{cory}

\begin{pf}
Let $B$ be a \grb as above. It follows from \cref{ribthm} and the implication (3$\Rightarrow$4) in \cref{ribequiv} that the matrix $D_B$ is connected. As explained in the discussion following \cref{ribequiv}, this implies that the decomposition matrix of $B$ is connected, so that $B$ is a single $p$-block of $\cenalgk$.
\end{pf}

\subsection{Belt blocks}\label{secbelt}

For the rest of this section suppose that $p=n$. Our aim is to establish \cref{mainconj} for a \emph{\belbl} $B$, i.e.\ a \conjbl of $\cenalgk$ whose \cench equals $\zpz=\{\widebar0,\dots,\widebar{p-1}\}$.

As in \cref{secrib}, we first replace $B$ (again regarded as a subset of $\calp_{l,m}$) with its set $\calb$ of \confs. We keep going with the same notation established in the preface of the chapter talking about the \conncomps of a \conf. Now the \gribs are ribbons. Unfortunately in this situation we cannot mimic at all the arguments we have used for \ribbls. Indeed, even if it is still true that the \forch of a disconnected \conf is the sum of the \forchs of its \refts, a basic step of the reasoning in \cref{secrib}, namely \cref{disjch}, is no longer true in this setting. This means that, in general, different \ribs may share composition factors.

\begin{eg}
Let $p=5$. The \forchs of the following two \ribs share two summands: $(1,2,0,3,4)$ and $(1,2,3,0,4)$.
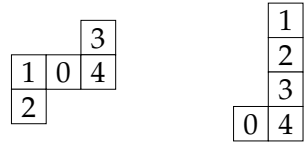
\[
\gyoung(::;3,104,2)\qquad\qquad\gyoung(:;1,:;2,:;3,04)
\]
\end{eg}

To fix this inconvenience, we need to introduce a functional 'refinement' of the set of \ribs. To do this, we begin with the introduction of a graph analogous to that in \cref{secrib}. We define a \emph{(circular) \arrg} to be a directed graph on the set $\zpz$ such that every arrow has the form $i\rightarrow i+1$ or $i+1\rightarrow i$ for some $i$ and the graph is not a directed cycle. An \arrg is a \emph{\belt} if for every $i$ there is either an arrow $i\rightarrow i+1$ or an arrow $i+1\rightarrow i$.

We will draw \arrgs with $\widebar0,\widebar1,\dots,\widebar{p-1}$ in clockwise order, omitting bars. So we may say that the arrows of an \arrg are all oriented clockwise to mean that they all have the form $i\to i+1$.

\begin{eg}
Let $p=7$. The following are two \arrgs, the one on the right being a \belt.
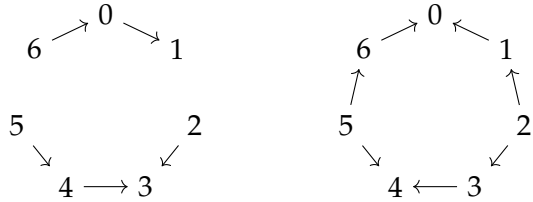
\[
\begin{tikzpicture}[scale=1.2]
\foreach\x in{0,1,2,3,4,5,6}\draw(90-\x*360/7:1)node(\x){$\x$};
\draw[->](6)--(0);
\draw[->](0)--(1);
\draw[->](2)--(3);
\draw[->](4)--(3);
\draw[->](5)--(4);
\end{tikzpicture}
\qquad\qquad
\begin{tikzpicture}[scale=1.2]
\foreach\x in{0,1,2,3,4,5,6}\draw(90-\x*360/7:1)node(\x){$\x$};
\draw[->](6)--(0);
\draw[->](1)--(0);
\draw[->](2)--(1);
\draw[->](2)--(3);
\draw[->](3)--(4);
\draw[->](5)--(4);
\draw[->](5)--(6);
\end{tikzpicture}
\]
\end{eg}
We observe that every \conf $X$ in a \belbl is uniquely associated to an \arrg $\arrX$ through the same procedure explained in \cref{secrib}. For example, the \arrg on the left of the above example corresponds to the \conf below.
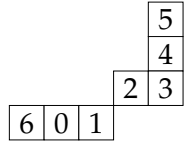
\[
\gyoung(::::;5,::::;4,:::;23,601)
\] 

Note that in this case the procedure does not give rise to a bijection between all possible \confs and all possible \arrgs. In particular, no \conf corresponds to a \belt. If $\Ga$ is an \arrg, we define its \forch $\ch{\Ga}$ to be the formal sum of all permutations $\bfb=(b_0,\dots,b_{p-1})$ of $\zpz$ with the property that $i$ appears before [resp.\ after] $i+1$ for every arrow $i\rightarrow i+1$ [resp.\ $i+1\rightarrow i$].
Note that if $X$ is a \conf, then $\ch{\arrX}=\ch{X}$.

Our next task is to build for every belt $\Ga$ a $\calh^{k}_p$-module whose \forch is $\ch{\Ga}$. Our construction is inspired by Young's orthogonal form, and as far as we can tell it is new. We take a vector space $M_\Ga$ with basis $\{e_{\bfb}\}$ labelled by the different summands of $\ch\Ga$.  If $\bfb=(b_0,\dots,b_{p-1})$ is a summand of $\ch\Ga$, then for each $k\in\{1,\dots,p-1\}$ we define $\sigma_k\bfb=(b_0,\dots,b_{k+1},b_k,\dots,b_{p-1})$ . We define an action of $\calh^k_p$ on $M_\Ga$ as follows. Fix a summand $\bold a$ of $\ch{\Ga}$. We set
\[
z_k\,\ee_{\bfb}=b_k\,\ee_{\bfb}
\]
and
\[
s_k\,\ee_{\bfb} =\begin{cases}
\frac{1}{b_{k+1}-b_k}\,\ee_{\bfb}+\ee_{\sigma_k\bfb} & \text{if } b_k \text{ appears \emph{after} } b_{k+1} \text{ in } \bfa,\\[3pt]
\frac{1}{b_{k+1}-b_k}\,\ee_{\bfb}+\left(1-\frac{1}{(b_{k+1}-b_k)^2}\right)\,\ee_{\sigma_k\bfb} & \text{if } b_k \text{ appears \emph{before} } b_{k+1} \text{ in } \bfa.
\end{cases}
\]

\begin{thm}\label{beltmod}
With the above rule $M_\Ga$ is a $\calh^{k}_p$-module with \forch $\ch{\Ga}$.

\begin{proof}
We just need to prove the first statement; the second part is then obvious from the definition of the \forch of a $\calh^{k}_p$-module.

We prove that $M_\Ga$ is an $\calh_p^k$-module by showing that the above defined rules respect the defining relations of $\calh^{k}_p$ shown in \cref{defdaha}. Relations $(1)$ and $(3)$ are trivially satisfied and need no further investigations. Therefore we only consider the \emph{braid} relations in $(2)$ and the relations between the polynomial and the Coxeter generators in $(4)$. 

We start from the braid relations. These divide into three types depending on the integers $m_{i,j}$ such that $(s_is_j)^{m_{i,j}}=1$. For each relation several cases arise depending on the positions of the involved entries of $\bfb$ in the fixed summand $\bfa$ of $\ch{\Ga}$. We only analyse one case for each relation, the others being similar.

\begin{itemize}
\item $s_k^2=1$. Suppose that $b_k$ appears before $b_{k+1}$ in $\bfa$. We have that \begin{align*}
s_k\ee_{\bfb}&=\frac{1}{b_{k+1}-b_k}\ee_{\bfb}+\biggr(1-\frac{1}{(b_{k+1}-b_k)^2}\biggl)\ee_{\sigma_k\bfb}\ \ \Rightarrow \\
s_k^2\ee_{\bfb}&=\frac{1}{(b_{k+1}-b_k)^2}\ee_{\bfb}+\frac{1}{b_{k+1}-b_k}\biggr(1-\frac{1}{(b_{k+1}-b_k)^2}\biggl)\ee_{\sigma_k\bfb} \\
&-\frac{1}{b_{k+1}-b_k}\biggr(1-\frac{1}{(b_{k+1}-b_k)^2}\biggl)\ee_{\sigma_k\bfb}+\biggr(1-\frac{1}{(b_{k+1}-b_k)^2}\biggl)\ee_{\bfb} \\ &=\ee_{\bfb}.
\end{align*}
\item $s_js_k=s_ks_j$ with $|k-j|>1$. Suppose that $b_k$ appears after $b_{k+1}$ and that $b_j$ appears before $b_{j+1}$ in $\bfa$. For the left-hand side we have
\begin{align*}
s_k\ee_{\bfb}&=\frac{1}{b_{k+1}-b_k}\ee_{\bfb}+\ee_{\sigma_k\bfb}\ \ \Rightarrow \\
s_js_k\ee_{\bfb}&=\frac{1}{(b_{k+1}-b_k)(b_{j+1}-b_j)}\ee_{\bfb}+\frac{1}{b_{k+1}-b_k}\biggl(1-\frac{1}{(b_{j+1}-b_j)^2}\biggr)\ee_{\sigma_j\bfb} \\
&+\frac{1}{b_{j+1}-b_j}\ee_{\sigma_k\bfb}+\biggl(1-\frac{1}{(b_{j+1}-b_j)^2}\biggr)\ee_{\sigma_j\sigma_k\bfb},
\end{align*}
while for the right-hand side we have
\begin{align*}
s_j\ee_{\bfb}&=\frac{1}{b_{j+1}-b_j}\ee_{\bfb}+\biggl(1-\frac{1}{(b_{j+1}-b_j)^2}\biggr)\ee_{\sigma_j\bfb}\ \ \Rightarrow \\
s_ks_j\ee_{\bfb}&=\frac{1}{(b_{k+1}-b_k)(b_{j+1}-b_j)}\ee_{\bfb}+\frac{1}{b_{j+1}-b_j}\ee_{\sigma_k\bfb} \\
&+\frac{1}{b_{k+1}-b_k}\biggl(1-\frac{1}{(b_{j+1}-b_j)^2}\biggr)\ee_{\sigma_j\bfb}+\biggl(1-\frac{1}{(b_{j+1}-b_j)^2}\biggr)\ee_{\sigma_k\sigma_j\bfb}.
\end{align*}
Comparing the coefficients of the different basis elements we easily gain the proof.
\item $s_ks_{k+1}s_k=s_{k+1}s_ks_{k+1}$. Suppose that $b_{k+1}$ precedes $b_{k+2}$ which precedes $b_k$ in $\bfa$. The left-hand and right-hand sides are as follows
\begin{align*}
s_k\ee_{\bfb}&=\frac{1}{b_{k+1}-b_k}\ee_{\bfb}+\ee_{\sigma_k\bfb}\ \ \Rightarrow \\
s_{k+1}s_k\ee_{\bfb}&=\frac{1}{(b_{k+1}-b_k)(b_{k+2}-b_{k+1})}\ee_{\bfb}+\frac{1}{b_{k+1}-b_k}\biggl(1-\frac{1}{(b_{k+2}-b_{k+1})^2}\biggr)\ee_{\sigma_{k+1}\bfb} \\
&+\frac{1}{b_{k+2}-b_k}\ee_{\sigma_k\bfb}+\ee_{\sigma_{k+1}\sigma_k\bfb}\ \ \Rightarrow \\
s_ks_{k+1}s_k\ee_{\bfb}&=\frac{1}{(b_{k+1}-b_k)^2(b_{k+2}-b_{k+1})}\ee_{\bfb}+\frac{1}{(b_{k+1}-b_k)(b_{k+2}-b_{k+1})}\ee_{\sigma_k\bfb} \\
&+\frac{1}{(b_{k+1}-b_k)(b_{k+2}-b_k)}\biggl(1-\frac{1}{(b_{k+2}-b_{k+1})^2}\biggr)\ee_{\sigma_{k+1}\bfb}\\
&+\frac{1}{b_{k+1}-b_k}\biggl(1-\frac{1}{(b_{k+2}-b_{k+1})^2}\biggr)\ee_{\sigma_k\sigma_{k+1}\bfb} \\
&-\frac{1}{(b_{k+2}-b_k)(b_{k+1}-b_k)}\ee_{\sigma_k\bfb}+\frac{1}{b_{k+2}-b_k}\biggl(1-\frac{1}{(b_{k+1}-b_k)^2}\biggr)\ee_{\bfb}\\
&+\frac{1}{b_{k+2}-b_{k+1}}\ee_{\sigma_{k+1}\sigma_k\bfb}+\biggl(1-\frac{1}{(b_{k+2}-b_{k+1})^2}\biggr)\ee_{\sigma_k\sigma_{k+1}\sigma_k\bfb}
\end{align*}
and
\begin{align*}
s_{k+1}\ee_{\bfb}&=\frac{1}{b_{k+2}-b_{k+1}}\ee_{\bfb}+\biggl(1-\frac{1}{(b_{k+2}-b_{k+1})^2}\biggr)\ee_{\sigma_{k+1}\bfb}\ \ \Rightarrow \\
s_ks_{k+1}\ee_{\bfb}&=\frac{1}{(b_{k+2}-b_{k+1})(b_{k+1}-b_k)}\ee_{\bfb}+\frac{1}{b_{k+2}-b_{k+1}}\ee_{\sigma_k\bfb}\\
&+\frac{1}{b_{k+2}-b_k}\biggl(1-\frac{1}{(b_{k+2}-b_{k+1})^2}\biggr)\ee_{\sigma_{k+1}\bfb}+\biggl(1-\frac{1}{(b_{k+2}-b_{k+1})^2}\biggr)\ee_{\sigma_k\sigma_{k+1}\bfb}\ \ \Rightarrow \\
s_{k+1}s_ks_{k+1}\ee_{\bfb}&=\frac{1}{(b_{k+2}-b_{k+1})^2(b_{k+1}-b_k)}\ee_{\bfb}\\
&+\frac{1}{(b_{k+2}-b_{k+1})(b_{k+1}-b_k)}\biggl(1-\frac{1}{(b_{k+2}-b_{k+1})^2}\biggr)\ee_{\sigma_{k+1}\bfb}\\
&+\frac{1}{(b_{k+2}-b_k)(b_{k+2}-b_{k+1})}\ee_{\sigma_k\bfb}+\frac{1}{b_{k+2}-b_{k+1}}\ee_{\sigma_{k+1}\sigma_k\bfb} \\
&-\frac{1}{(b_{k+2}-b_k)(b_{k+2}-b_{k+1})}\biggl(1-\frac{1}{(b_{k+2}-b_{k+1})^2}\biggr)\ee_{\sigma_{k+1}\bfb}\\
&+\frac{1}{b_{k+2}-b_k}\biggl(1-\frac{1}{(b_{k+2}-b_{k+1})^2}\biggr)\ee_{\bfb}\\
&+\frac{1}{b_{k+1}-b_k}\biggl(1-\frac{1}{(b_{k+2}-b_{k+1})^2}\biggr)\ee_{\sigma_k\sigma_{k+1}\bfb}+\biggl(1-\frac{1}{(b_{k+2}-b_{k+1})^2}\biggr)\ee_{\sigma_{k+1}\sigma_k\sigma_{k+1}\bfb}.
\end{align*}
The coefficients of $\ee_{\sigma_k\sigma_{k+1}\bfb},\ee_{\sigma_{k+1}\sigma_k\bfb}$ and $\ee_{\sigma_k\sigma_{k+1}\sigma_k\bfb}=\ee_{\sigma_{k+1}\sigma_k\sigma_{k+1}\bfb}$ are the same on both sides. The coefficients of $\ee_{\bfb},\ee_{\sigma_k\bfb}$ and $\ee_{\sigma_{k+1}\bfb}$ turn out to be equal from easy calculations comparing the two above formulas.
\end{itemize}

We finally focus on the relation $(4)$: $s_kz_k=z_{k+1}s_k-1$. Suppose that $b_k$ precedes $b_{k+1}$ in $\bfa$. We evaluate the action of both sides of $(4)$ on $\ee_{\bfb}$. On the left-hand we have
\begin{align*}
z_k\ee_{\bfb}&=b_k\ee_{\bfb}\ \ \Rightarrow \\
s_kz_k\ee_{b}&=\frac{b_k}{b_{k+1}-b_k}\ee_{\bfb}+b_k\biggl(1-\frac{1}{(b_{k+1}-b_k)^2}\biggr)\ee_{\sigma_k\bfb},
\end{align*}
while on the right we have
\begin{align*}
s_k\ee_{\bfb}&=\frac{1}{b_{k+1}-b_k}\ee_{\bfb}+\biggl(1-\frac{1}{(b_{k+1}-b_k)^2}\biggr)\ee_{\sigma_k\bfb}\ \ \Rightarrow \\
(z_{k+1}s_k-1)\ee_{\bfb}&=\frac{b_{k+1}}{b_{k+1}-b_k}\ee_{\bfb}+b_k\biggl(1-\frac{1}{(b_{k+1}-b_k)^2}\biggr)\ee_{\sigma_k\bfb}-\ee_{\bfb}.
\end{align*}
It is now straightforward to check that the coefficients of $\ee_{\bfb}$ and $\ee_{\sigma_k\bfb}$ are equal on both sides. This ends the proof.
\end{proof}
\end{thm}

\cref{beltmod} strengthens the above discussion and shows that \emph{every} \arrg corresponds to a $\calh^{k}_p$-module. In particular, we can talk about the composition factors of a belt, meaning the composition factors of the corresponding $\calh^k_p$-module. Later on we may look at the set of \confs $\calb$ as a subset of the set of all possible \arrgs. 

As in \cref{secrib} we rely on the notion of distance between \confs. Given \confs $X$ and $Y$, we define $\Dist{X}{Y}$ and $\dist{X}{Y}$ as in \cref{secrib}.

\begin{eg}
Let $p=7$, and let $X$ and $Y$ be the \confs below. From the corresponding \arrgs, we see that $\Dist{X}{Y}=\{\widebar0,\widebar4\}$, and therefore $\dist{X}{Y}=2$.
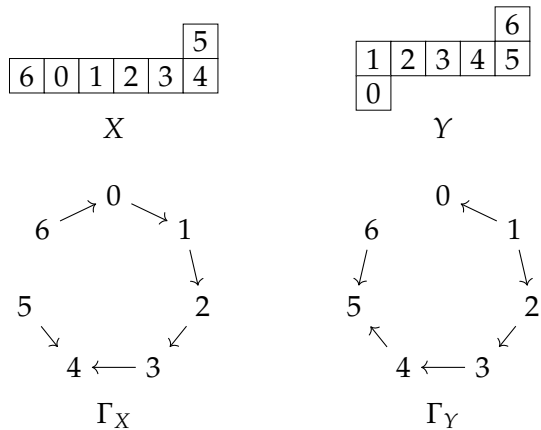
\[
\begin{array}{c@{\qquad\qquad}c}
\gyoung(:::::;5,601234)
&
\gyoung(::::;6,12345,0)
\\
X&Y
\\[10pt]
\begin{tikzpicture}[scale=1.2]
\foreach\x in{0,1,2,3,4,5,6}\draw(90-\x*360/7:1)node(\x){$\x$};
\draw[->](6)--(0);
\draw[->](0)--(1);
\draw[->](1)--(2);
\draw[->](2)--(3);
\draw[->](3)--(4);
\draw[->](5)--(4);
\end{tikzpicture}
&
\begin{tikzpicture}[scale=1.2]
\foreach\x in{0,1,2,3,4,5,6}\draw(90-\x*360/7:1)node(\x){$\x$};
% \draw[->](6)--(0);
\draw[->](1)--(0);
\draw[->](1)--(2);
\draw[->](2)--(3);
\draw[->](3)--(4);
\draw[->](4)--(5);
\draw[->](6)--(5);
\end{tikzpicture}
\\
\Ga_X&\Ga_Y
\end{array}
\]

\end{eg}

We can now state a variant of \cref{disjch}.

\begin{lemma}\label{disjch2}
Consider $\bfa=(a_0,\dots,a_{p-1})$ where $a_0,\dots,a_{p-1}$ equal $\widebar0,\dots,\widebar{p-1}$ in some order. Then $\bfa$ appears in the \forch of exactly one \belt.

\begin{pf}
This result follows similarly to \cref{disjch}. Take an empty \arrg and fill its edges according to the relative position between the entries of $\bfa$ until a \belt $X$ is reached. By construction $\bfa$ is a term of the \forch $\ch{X}$. If a different \belt $Y$ is considered, then $\Dist{X}{Y}\ne\emptyset$ and repeating the last argument in the proof of \cref{disjch}, we see that $\ch{X}$ and $\ch{Y}$ have no terms in common.
\end{pf}

\end{lemma}

\begin{cory}
Different \belts have no composition factors in common.
\end{cory}

At this point we can mimic the arguments in \cref{secrib} using \belts in place of maximal \confs.

\begin{defn} Let $X$ be a \conf. A \belt $Y$ is a \emph{\reft} of $X$ if $Y$ can be obtained by adding edges to~$X$. 
\end{defn}

\begin{cory}\label{refs2}
Let $X$ be a \conf. The \forch of $X$ is the sum of the \forchs of the \refts of $X$.
\begin{pf}
Given $Y$ a \reft of $X$, every term of $\ch{Y}$ is also a term of $\ch{X}$ (just because $X$ is obtained forgetting some arrows of $Y$). Conversely, each term of $\ch{X}$ is a term of the \forch of a unique \reft of $X$: for every $j$ and $j+1\ppmod p$ that are not linked by an arrow in $X$, the relative order of these entries in the chosen term of $\ch{X}$ indicate how they need to be linked to find the desired \belt.
\end{pf}
\end{cory}

As in \cref{compfact}, \cite[Theorem 5.3.1]{klbook} ensures the following.

\begin{cory}\label{beltref}
If $X$ is a \conf, then the composition factors of $X$ are the composition factors of its \refts.

\end{cory}

We are know ready to establish a result which gives us a way to recognise when two Specht $\cenalgk$-modules share a composition factor.

\begin{thm}\label{beltequiv}
Let $X$ and $Y$ be \confs. The following are equivalent:
\begin{enumerate}
\item $X$ and $Y$ have a composition factor in common;
\item The \forchs of $X$ and $Y$ have a term in common;
\item $\dist{X}{Y}=0$ and $\arrX\cup\arrY$ is not a directed cycle;
\item There is a \belt which is a \reft of both $X$ and $Y$.
\end{enumerate}
\end{thm}

\begin{rmk}
Before proving \cref{beltequiv} we dwell a moment on the meaning of condition $(3)$. Since $\dist{X}{Y}=0$, the overlapping $\arrX\cup\arrY$ is an admissible \arrg. We have that $\arrX\cup\arrY$ is a directed cycle if and only if all arrows in $\arrX$ and $\arrY$ are anticlockwise (or all clockwise) and there is no $k\in\zpz$ such that $\bxk$ is simultaneously the foot node of a \conncomp of $X$ and a \conncomp of $Y$.
\end{rmk}

\begin{eg}
Let $p=7$, and let $X$ and $Y$ be the \confs given below. All arrows in $\Ga_X$ and $\Ga_Y$ are oriented clockwise (because neither $X$ nor $Y$ has a node immediately above another). Moreover, $X$ and $Y$ do not have foot nodes of the same residue. So $\Ga_X\cup\Ga_Y$ is a directed cycle. As a consequence, the formal characters of $X$ and $Y$ have no term in common, so $X$ and $Y$ do not share a composition factor.
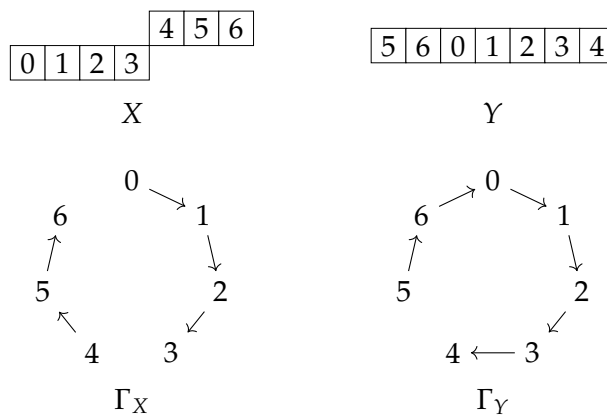
\[
\begin{array}{c@{\qquad\qquad}c}
\gyoung(::::;456,0123)
&
\young(5601234)
\\[12pt]
X&Y
\\[10pt]
\begin{tikzpicture}[scale=1.2]
\foreach\x in{0,1,2,3,4,5,6}\draw(90-\x*360/7:1)node(\x){$\x$};
\draw[->](0)--(1);
\draw[->](1)--(2);
\draw[->](2)--(3);
\draw[->](4)--(5);
\draw[->](5)--(6);
\end{tikzpicture}
&
\begin{tikzpicture}[scale=1.2]
\foreach\x in{0,1,2,3,4,5,6}\draw(90-\x*360/7:1)node(\x){$\x$};
\draw[->](0)--(1);
\draw[->](1)--(2);
\draw[->](2)--(3);
\draw[->](3)--(4);
\draw[->](6)--(0);
\draw[->](5)--(6);
\end{tikzpicture}
\\
\Ga_X&\Ga_Y
\end{array}
\]
\end{eg}

\begin{pf}[Proof of \cref{beltequiv}]
\begin{description}[beginthm]
\item[\rm (1$\Rightarrow2$)]
This is trivial because the \forch of a $\calh_p^k$-module is the sum of the \forchs of its composition factors.

\item[\rm (2$\Rightarrow3$)]
Suppose that $\ch{X}$ and $\ch{Y}$ have a term in common. Then it immediately follows that $\dist{X}{Y}=0$. Suppose for a contradiction that $\arrX\cup\arrY$ is a clockwise directed cycle and set
\[
F_X\vcentcolon=\{k\in\cch{B}\ |\ \bxk \text{ is the foot node of a \conncomp of } X\}
\]
and \[
F_Y\vcentcolon=\{k\in\cch{B}\ |\ \bxk \text{ is the foot node of a \conncomp of } Y\}.
\]
(For example, for the $7$-shapes in the example above, $F_X=\{\widebar0,\widebar4\}$ and $F_Y=\{\widebar5\}$.) The remark above shows that $F_X\cap F_Y=\emptyset$. Now since every term of $\ch{X}$ has an element of $F_X$ as first component and every term of $\ch{Y}$ has an element of $F_Y$ as first component, it follows that $\ch{X}$ and $\ch{Y}$ have no terms in common and this contradicts the hypothesis. (If the arrows in $\arrX\cup\arrY$ are all anticlockwise, the proof follows similarly replacing "foot" with "hand" in the definitions of $F_X$ and $F_Y$ above.)

\item[\rm (3$\Rightarrow4$)]
By the hypothesis the union $\arrX\cup\arrY$ is an \arrg. This can be refined to a belt by adding in arrows if necessary (ensuring that these arrows are not all oriented the same way as all the existing arrows). This belt is a \reft of both $X$ and~$Y$.

\item[\rm (4$\Rightarrow1$)]
This follows from \cref{beltref}.\qedhere
\end{description}
\end{pf}

In fact, we can prove that the modules $M_\Ga$ are simple. We do this next (although we do not actually need it for our main result).
% It seems likely that the set of \belts gives a family of simple $\calh_p^k$-modules. Despite this, we will not focus on this question in the present paper. We record this a conjecture to inspire further research.

% \begin{conj}
% 	Let $\Ga$ be a belt. Then the $\calh^{k}_p$-module $M_{\Ga}$ is simple.
% \end{conj}

\begin{cory}
Let $\Ga$ be a belt. Then the $\calh^{k}_p$-module $M_{\Ga}$ is simple.
\end{cory}

\begin{pf}
Let $\bfB$ be the set of all permutations $\bfb$ appearing in $\ch\Ga$. Given $\bfc\in\bfB$, define
\[
z_\bfc=\prod_{k=1}^p\prod_{\substack{i\in\zpz\\i\neq\bfc_k}}(z_k-i).
\]
It is easy to see that $z_\bfc\ee_\bfc\neq0$, while $z_\bfc\ee_\bfb=0$ for every $\bfb\neq\bfc$.

Now suppose $N$ is a non-zero submodule of $M_\Ga$. Take a non-zero element $n=\sum_{\bfb\in\bfB}\la_\bfb\ee_\bfb$ of $N$, and choose $\bfc$ such that the coefficient $\la_\bfc$ is non-zero. Then $z_\bfc n$ is a non-zero scalar multiple of $\ee_\bfc$, so $\ee_\bfc\in N$.

Now suppose $\bfd$ is another element of $\bfB$, and write $\bfd=\si_{k_1}\dots\si_{k_r}\bfc$ with $r$ as small as possible. Note then that all the intermediate terms $\si_{k_j}\dots\si_{k_r}\bfc$ will also lie in $\bfB$. Now the formula defining the action of $s_k$ on $M_\Ga$ shows that $\ee_\bfd$ appears with non-zero coefficient in $s_{k_1}\dots s_{k_r}\ee_\bfc$. Hence $z_\bfd s_{k_1}\dots s_{k_r}\ee_\bfc$ is a non-zero multiple of $\ee_\bfd$ lying in $N$, so that $\ee_\bfd\in N$. So $N=M_\Ga$.
\end{pf}

%At this point the reader might expect that the set of \belts gives a family of simple $\calh_p^k$-modules. We suspect that this is true, but we will not focus on the question in the present paper.

Now we construct an approximation to the decomposition matrix as in \cref{secrib}: we construct the matrix
\[
D_B=(d_{X,Y})_{X\in\calb,Y\in\mathrm{Belt}},
\]
where $d_{X,Y}=1$ if $Y$ is a refinement of $X$, and $0$ otherwise. As in \cref{secrib}, showing that $D_B$ is a connected matrix is sufficient to show that $B$ is a block.

\begin{defn}
Let $X,Y$ be \confs in $\calb$. We say that $X$ and $Y$ are \emph{$p$-linked} if there exists a series of \confs $X=X_1,X_2,\dots,X_{t-1},X_t=Y$ in $\calb$ such that $X_j$ and $X_{j+1}$ satisfy \emph{any} of the equivalent conditions in \cref{beltequiv}, for every $1\le j\le t-1$.
\end{defn}

We now move to the heart of this section. First of all we show that we can reduce to the case where two \confs in a \belbl have distance zero.

\begin{propn}\label{prop2belt}
Let $X$ and $Y$ be \confs in $\calb$ with $\dist{X}{Y}>0$. There exists a \conf $Z\in\calb$ such that $\dist{X}{Z}<\dist{X}{Y}$ and $\dist{Y}{Z}<\dist{X}{Y}$.

\begin{pf}

Fix $i\in\Dist{X}{Y}$. We use all the notation introduced for proving \cref{grbthm}. Let $\X$ and $\Y$ be \skewps in $B$ having \confs $X$ and $Y$, respectively.  Assume without loss of generality that $x_i$ is a space in $\abx$ and $y_i$ is a bead in $\aby$. Let $k\in\{i+1,\dots,i+p-1\}$ be the integer defined in \cref{lemma1}. If $k$ exists, then we can follow the proof of \cref{grbthm} to modify $\mu_X$ or $\la_Y$ depending on the situation, and to show the existence of a \skewp in $B$ with an appropriate \conf. If $k$ does not exist, it follows that for every $j\in\{i+1,\dots,i+p-1\}$ with $j\notin\Dist{X}{Y}$, 
\begin{itemize}
	\item if $x_{j}$ is a bead or a right half-bead in $\abx$, then $\mu_X$ has no addable $(k-i)$-hooks at runner $i$,
\end{itemize}	
and
\begin{itemize}
	\item if $y_k$ is a space or a right half-bead in $\aby$, then $\la_Y$ has no removable $(k-i)$-hooks at runner~$k$.
\end{itemize}

\begin{figure}[hbt]
\[
\begin{tikzpicture}[label distance=-4.5pt]
	\foreach\g in{4.5}{
		\foreach\x in {0,1.5,3,4.5}\draw(\x,-.5)--++(0,8)++(\g+4.5,0)--++(0,-8);
		
		%beads
		\draw(0,1)node[label=below left:${b}$](b){\ };\draw(b)node{$\abacus(b)$};
		\draw(0,2)node(1){\ };\draw(1)node{$\abacus(b)$};
		\draw(0,3)node[label=below left:${c_s}$](cs){\ };\draw(cs)node{$\abacus(b)$};\begin{scope}[scale=.6]\bluertbll05\end{scope}
		\draw(3,1)node[label=below left:${c_{s-1}}$](cs-1){\ };\draw(cs-1)node{$\abacus(b)$};\begin{scope}[scale=.6]\bluertbll5{5/3}\end{scope}
		\draw(7.5+\g,1)node[label=below left:${c_1}$](c1){\ };\draw(c1)node{$\abacus(b)$};\begin{scope}[scale=.6]\bluertbll{12.5+5*\g/3}{5/3}\end{scope}
		\draw(4.5+\g,3)node[label=below left:${c_2}$](c2){\ };\draw(c2)node{$\abacus(b)$};\begin{scope}[scale=.6]\bluertbll{7.5+5*\g/3}5\end{scope}
		\draw(1.5,3)node(2){\ };\draw(2)node{$\abacus(b)$};
		\draw(4.5,1)node(3){\ };\draw(3)node{$\abacus(b)$};
		\draw(6+\g,3)node(4){\ };\draw(4)node{$\abacus(b)$};
		\draw(9+\g,1)node(5){\ };\draw(5)node{$\abacus(b)$};
		
		%spaces
		\draw(0,0)node[label=below right:${d_1}$](d1){\ };\draw(d1)node{$\abacus(n)$};
		\draw(7.5+\g,3)node[label=below right:${d_2}$](d2){\ };\draw(d2)node{$\abacus(n)$};
		\draw(4.5+\g,2)node[label=below right:${d_3}$](d3){\ };\draw(d3)node{$\abacus(n)$};
		\draw(3,3)node[label=below right:${d_s}$](ds){\ };\draw(ds)node{$\abacus(n)$};
		\draw(0,4)node(6){\ };\draw(6)node{$\abacus(n)$};
		\draw(0,5)node(61){\ };\draw(61)node{$\abacus(n)$};
		\draw(0,6)node(62){\ };\draw(62)node{$\abacus(n)$};
		\draw(0,7)node(63){\ };\draw(63)node{$\abacus(n)$};
		\draw(3,6)node(7){\ };\draw(7)node{$\abacus(n)$};
		\draw(4.5+\g,6)node(8){\ };\draw(8)node{$\abacus(n)$};
		\draw(7.5+\g,7)node(9){\ };\draw(9)node{$\abacus(n)$};

		%right half-beads
		\draw(9+\g,7)node(j1){\ };%\draw(j1)node{$\abacus(M)$};
		\begin{scope}[scale=.6]\redrtbll{15+5*\g/3}{35/3}\end{scope}
		\draw(6+\g,6)node(j2){\ };%\draw(j2)node{$\abacus(M)$};
		\begin{scope}[scale=.6]\redrtbll{10+5*\g/3}{10}\end{scope}
		\draw(4.5,6)node(j3){\ };%\draw(j3)node{$\abacus(M)$};
		\begin{scope}[scale=.6]\redrtbll{7.5}{10}\end{scope}
		\draw(1.5,4)node(j4){\ };%\draw(j4)node{$\abacus(M)$};
		\begin{scope}[scale=.6]\redrtbll{2.5}{20/3}\end{scope}
		
		%left half-beads
		\draw(6+\g,7)node(k1){\ };\draw(k1)node{$\abacus(L)$};
		\draw(4.5,5)node(k2){\ };\draw(k2)node{$\abacus(L)$};
		\draw(1.5,6)node(k3){\ };\draw(k3)node{$\abacus(L)$};
		\draw(9+\g,5)node(k4){\ };\draw(k4)node{$\abacus(L)$};
		
		%dots
		\draw(-.75,0)node{\dots};
		\draw(-.75,4)node{\dots};
		\draw(.75,3)node{\dots};
		\draw(.75,4)node{\dots};
		\draw(2.25,3)node{\dots};
		\draw(2.25,6)node{\dots};
		\draw(3.75,1)node{\dots};
		\draw(3.75,6)node{\dots};
		\draw(5.25,1)node{\dots};
		\draw(5.25,5)node{\dots};
		\draw(3.75+\g,2)node{\dots};
		\draw(3.75+\g,6)node{\dots};
		\draw(5.25+\g,3)node{\dots};
		\draw(5.25+\g,6)node{\dots};
		\draw(6.75+\g,3)node{\dots};
		\draw(6.75+\g,7)node{\dots};
		\draw(8.25+\g,1)node{\dots};
		\draw(8.25+\g,7)node{\dots};
		\draw(9.75+\g,1)node{\dots};
		\draw(9.75+\g,5)node{\dots};
		
		%arrows
		\draw(c1)++(2,.6)coordinate(cc);
		\draw(j4)++(-2,.6)coordinate(dd);
		\draw(j2)++(-2,.6)coordinate(ee);
		\draw(cs-1)++(2,.6)coordinate(ff);
		\draw[->,thick,red](j1)to[out=150,in=30](k1);
		\draw[->,thick,red](j4)to[out=150,in=0](dd);
		\draw[->,thick,red](k4)++(1,.3)to[out=180,in=30](k4);
		\draw[->,thick,red](j2)to[out=150,in=0](ee);
		\draw[->,thick,red](k2)++(1,.3)to[out=180,in=30](k2);
		\draw[->,thick,red](j3)to[out=150,in=30](k3);
		
		\draw[->,thick,\myblue](c1)to[out=30,in=180](cc);
		\draw[->,thick,\myblue](d1)+(-1,.3)to[out=0,in=150](d1);
		\draw[->,thick,\myblue](c2)to[out=30,in=150](d2);
		\draw[->,thick,\myblue](d3)+(-1,.3)to[out=0,in=150](d3);
		\draw[->,thick,\myblue](cs)to[out=30,in=150](ds);
		\draw[->,thick,\myblue](cs-1)to[out=30,in=180](ff);
	}
\end{tikzpicture}
\]
\caption{$\Ab{X}$.}\label{bigbeltthmfig}
\end{figure}

We perform a modification of $\mu_X$ that comprises deleting \emph{all} the nodes in $\X$ and adding them in different positions. The illustrative abacus in \cref{bigbeltthmfig} illustrates the bead moves we are going to do (we arrange the runners such that runner $i$ is the leftmost).

The partition $\mu_X$ has an addable $p$-hook at runner $i$, this is at the largest bead position of runner $i$ of $\Ab{\mu_X}$, say $b\equiv i\ppmod p$. We can now emulate the construction in the proof of \cref{grbthm} to find integers $k_1,\dots,k_s$ such that $\sum_{1\le u\le s}k_u=p$, and positions $c_1,\dots,c_s,d_1,\dots,d_s$ of the abacus display such that $c_1,\dots,c_s$ are beads and $d_1,\dots,d_s$ are spaces in $\abx$. We find a partition $\mu_Z\vdash m$ obtained from $\mu_X$ removing the $p$ nodes in $\X$ and moving the right half of the bead at $c_v$ to the space at $d_v$ for all $v$. If $Z$ is the \conf of $\Xmod$, as in \cref{grbthm} the hypothesis on $k$ implies that $Z$ is the desired \conf.
\end{pf}
\end{propn}

\newcommand\bound{\path[clip](-1,-1.5)rectangle(7,7*\yoy);}

\begin{eg}
Let $m=119$, $l=112$ and $p=n=7$. Consider the \skewps
\begin{align*}
\X&=(20,18,17,16,12^4)\setminus(20,18,17,16,12^2,11,6),
\\
\Y&=(21^2,17,15,14,13,12,6)\setminus(20,16^2,15,14,13,12,6)
\end{align*}
in $\calp_{112,119}$ having $7$-shapes $X$ and $Y$, respectively. These belong to the same \conjbl of $\calc_{112,119}^{k}$ as can be easily checked (all the partitions involved have $7$-core $(13,7,3,2^2,1)\vdash28$). Below we show the \abacs with charge 21 and the \arrgs of $X$ and $Y$.

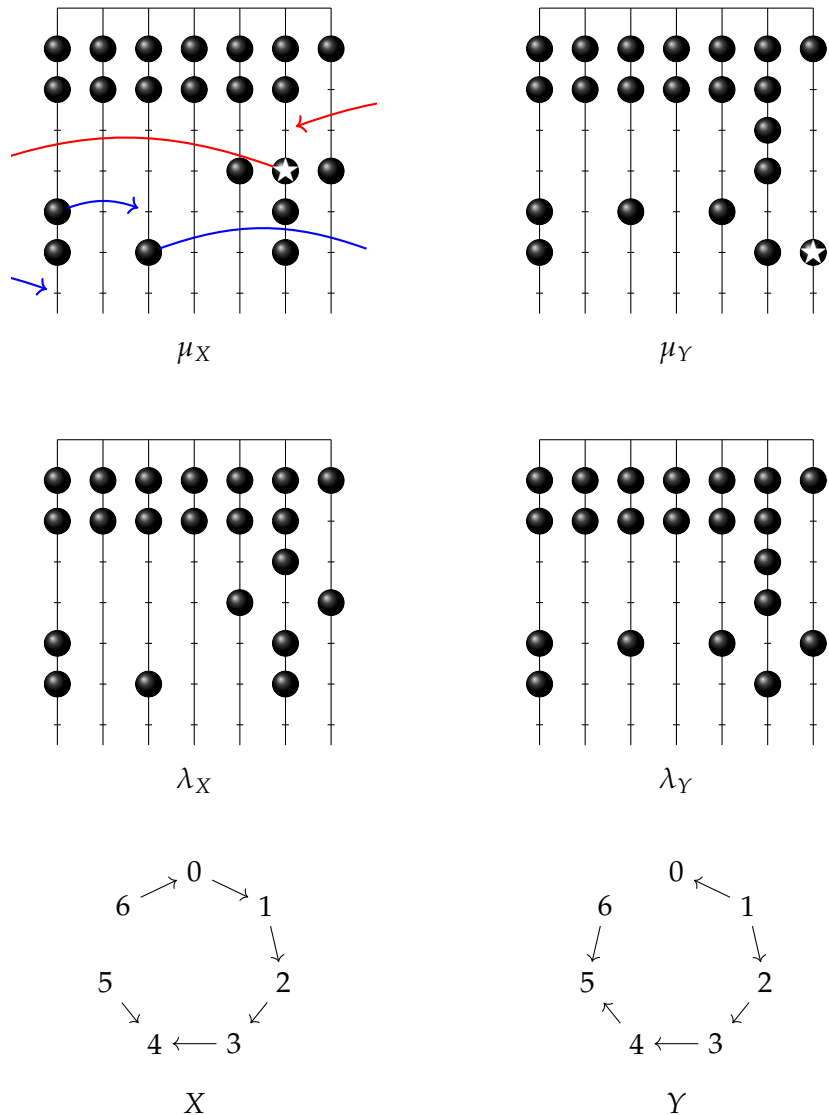
\begin{align*}
&X&&\begin{tikzpicture}[scale=.6,baseline=1.5cm]
\bound
\draw(0,5.5*\yoy)--++(6,0);
% \foreach\x in{0,1,2,3,4,5,6}{\bll\x6}
\foreach\x in{0,1,2,3,4,5}{\bll\x5}\spce65
\foreach\x in{0,1,2,3,4,6}{\spce\x4}\lfbll54
\foreach\x in{0,1,2,3}{\spce\x3}\bll43\bll63\rtbll53
\foreach\x in{1,2,3,4,6}{\spce\x2}\bll02\bll52
\foreach\x in{1,3,4,6}{\spce\x1}\bll01\bll21\bll51
\foreach\x in{0,1,2,3,4,5,6}{\spce\x0}
\redrtbll53
\bluertbll02
\bluertbll21
\draw(5,3*\yoy)node(a){\ };
\draw(5,4*\yoy)node(b){\ };
\draw(12,4*\yoy)node(aa){\ };
\draw(-2,3*\yoy)node(bb){\ };
\draw[->,red,thick](a)to[out=160,in=20](bb);
\draw[->,red,thick](aa)to[out=160,in=20](b);
\draw(2,1*\yoy)node(e){\ };
\draw(0,0*\yoy)node(f){\ };
\draw(-5,0*\yoy)node(ee){\ };
\draw(7,1*\yoy)node(ff){\ };
\draw[\myblue,thick](e)to[out=20,in=160](ff);
\draw[->,\myblue,thick](ee)to[out=20,in=160](f);
\draw(0,2*\yoy)node(g){\ };
\draw(2,2*\yoy)node(h){\ };
\draw[->,\myblue,thick](g)to[out=20,in=160](h);
\end{tikzpicture}
&&	\begin{tikzpicture}[scale=1.4,baseline=0cm]
		\foreach\x in{0,1,2,3,4,5,6}\draw(90-\x*360/7:1)node(\x){$\x$};
		\draw[->](6)--(0);
		\draw[->](0)--(1);
		\draw[->](1)--(2);
		\draw[->](2)--(3);
		\draw[->](3)--(4);
		\draw[->](5)--(4);
	\end{tikzpicture}
\\
&Y&&\begin{tikzpicture}[scale=.6,baseline=1.5cm]
\bound
\draw(0,5.5*\yoy)--++(6,0);
% \foreach\x in{0,1,2,3,4,5,6}{\bll\x6}
\foreach\x in{0,1,2,3,4,5}{\bll\x5}\spce65
\foreach\x in{0,1,2,3,4,6}{\spce\x4}\bll54
\foreach\x in{0,1,2,3,4,6}{\spce\x3}\bll53
\foreach\x in{1,3,5}{\spce\x2}\foreach\x in{0,2,4}{\bll\x2}\lfbll62
\foreach\x in{1,2,3,4}{\spce\x1}\bll01\bll51\rtbll61
\foreach\x in{0,1,2,3,4,5,6}{\spce\x0}
\end{tikzpicture}
&&
	\begin{tikzpicture}[scale=1.4,baseline=0cm]
		\foreach\x in{0,1,2,3,4,5,6}\draw(90-\x*360/7:1)node(\x){$\x$};
		\draw[->](1)--(0);
		\draw[->](1)--(2);
		\draw[->](2)--(3);
		\draw[->](3)--(4);
		\draw[->](4)--(5);
		\draw[->](6)--(5);
	\end{tikzpicture}
\end{align*}

We see that $\Dist{X}{Y}=\{\ol0,\ol4\}$ and hence $\dist{X}{Y}=2$. Let $i=\ol0$. It is easy to see that the integer $k$ does not exists. Following \cref{grbthm} and  \cref{prop2belt}, we perform the modification of $\mu_X$ as shown in the abacus above. We obtain the partition $\mu_Z=(22,21,17,16,14,12,11,6)\vdash119$. Denoting by $Z$ the $7$-shape of $\Xmod$, we see that $\dist{X}{Z}=1$ and $\dist{Y}{Z}=0$.
% 
% \begin{comment}
% We see that $\Dist{X}{Y}=\{\ol0,\ol4\}$ and hence $\dist{X}{Y}=2$. We follow the notation of \cref{grbthm}. Let $i=\widebar0$: position $x_0=21$ is a space in $\Ab{\mu_X}$ and position $y_0=35$ is a bead in $\Ab{\mu_Y}$. Here, for every $1\le k\le6$, no modifications of $\mu_X$ or $\la_Y$ are allowed by the procedure. Consider the addable $7$-hook at runner 0 of $\Ab{\mu_X}$ at position $a_1=35$. This gives $d_1=42$ from which follows that $k_1=5$ and $c_1=37$. The partition $\mu_X$ has an addable 2-hook at runner 0, this is at position $a_2=28$. Hence, $d_2=30$ and consequently $k_2=2$ and $c_2=a_2=28$. Then $s=2$ and we stop. 
% For the 'removing' part, since $X$ is a \rib, $t=1$ and $j_1=\widebar5$. We modify $\Ab{\mu_X}$ according to the rule given in \cref{grbthm} (as shown above), obtaining the partition $\mu_Z=(22,21,17,16,14,12,11,6)\vdash119$. Denoting by $Z$ the $7$-shape of $\Xmod$, we see that $\dist{X}{Z}=1$ and $\dist{Y}{Z}=0$.
% \end{comment}
% 
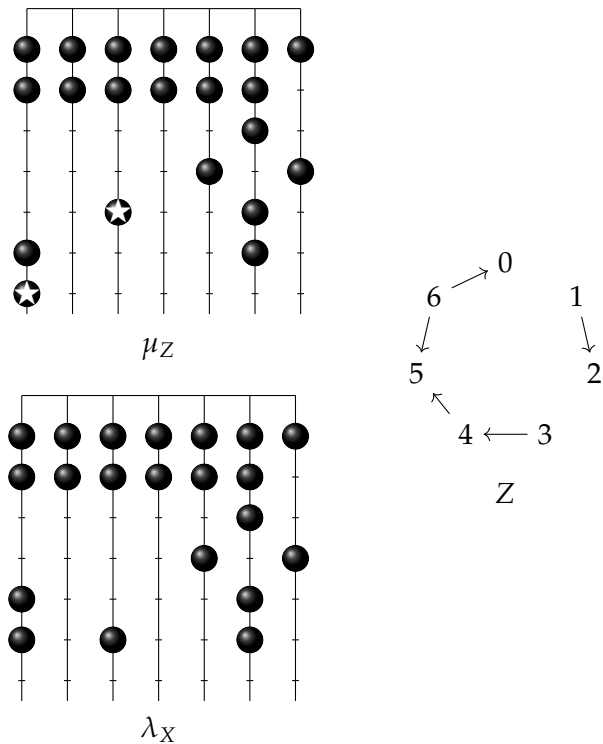
\begin{align*}
&Z&&\begin{tikzpicture}[scale=.6,baseline=1.5cm]
\bound
\draw(0,5.5*\yoy)--++(6,0);
% \foreach\x in{0,1,2,3,4,5,6}{\bll\x6}
\foreach\x in{0,1,2,3,4,5}{\bll\x5}\spce65
\foreach\x in{0,1,2,3,4,6}{\spce\x4}\bll54
\foreach\x in{0,1,2,3,5}{\spce\x3}\bll43\bll63
\foreach\x in{1,3,4,6}{\spce\x2}\lfbll02\rtbll22\bll52
\foreach\x in{1,3,4,6}{\spce\x1}\bll01\lfbll21\bll51
\foreach\x in{1,2,3,4,5,6}{\spce\x0}\rtbll00
\end{tikzpicture}
&&	\begin{tikzpicture}[scale=1.4,baseline=0cm]
		\foreach\x in{0,1,2,3,4,5,6}\draw(90-\x*360/7:1)node(\x){$\x$};
		\draw[->](1)--(2);
		\draw[->](3)--(4);
		\draw[->](4)--(5);
		\draw[->](6)--(5);
		\draw[->](6)--(0);
	\end{tikzpicture}
\end{align*}

\end{eg}

Now we use \cref{prop2belt} to show that the \conjbl $B$ is actually a $p$-block of $\cenalgk$. This is slightly more complicated than in \cref{secrib}, because it is no longer the case that two \confs at distance $0$ are $p$-linked. Take two \confs $X,Y\in\calb$ such that $\dist{X}{Y}=0$. By \cref{beltequiv}, $X$ and $Y$ are $p$-linked if $\arrX\cup\arrY$ is not a directed cycle. Hence from now on we suppose that $\arrX\cup\arrY$ is a directed cycle and we assume \wolog that all its arrows are oriented clockwise. The aim of the rest of this section is to prove that $X$ and $Y$ are $p$-linked. 

First of all we establish some notation. Every \skewp $\mulam$ in $B$ satisfies the hypothesis of \cref{beltcore} and hence $\core{\mu}=\core{\la}$. So there is a $p$-core $\ga$ and an integer $w$ such that
\[
\core\mu=\core\la=\ga,\qquad\mathrm{weight}_p(\la)=w,\quad\mathrm{weight}_p(\mu)=w+1
\]
for every \skewp $\mulam$ in $B$. Let $c\in\zpz$ be the $p$-residue of the right-most node in the first part of $\gamma$, i.e.\ $c=\widebar{\gamma_1-1}$. If $\gamma$ is the empty partition, we set $c\vcentcolon=\widebar{p-1}$.

Let $L$ be the $p$-block of $S_l$ whose partitions all have $p$-core $\gamma$ and $p$-weight $w$. For each $i\in\zpz$, let \[
L_i=\{\la\in L\ |\ \widebar{\la_1-1}=i\}.
\] We remark that some $L_i$ can be empty and that \[
L=\bigsqcup_{i\in\zpz}L_i.
\]
For every $\la\in L$, let $O_\la$ be the set of partitions $\mu\vdash m$ such that $\mulam\in B$ and all the arrows of the \conf of $\mulam$ are clockwise (i.e.\ $\mulam$ does not contain two nodes in the same column).

\begin{rmk} 
Suppose $\la_X,\la_Y\in L$, $\mu_X\in O_{\la_X}$ and $\mu_Y\in O_{\la_Y}$, and let $X$ and $Y$ be the \confs of $\X$ and $\Y$ respectively. Then $\dist{X}{Y}=0$, because all edges are oriented clockwise. Moreover, every pair $X,Y$ such that $\arrX\cup\arrY$ is a clockwise directed cycle arises in this way.
\end{rmk}

For every $i\in\zpz$, let $X_i$ be the (\rib) \conf
\[
\gyoungs({{1.4}},\ipo_2\hdts\imo i).
\]
For example, if $p=7$ and $i=4$, $X_4$ is the \conf $Y$ in the example in \cref{beltequiv}.

Let $\la\in L_i$ and let $\mu_\la=(\la_1+p,\la_2,\dots,\la_{l(\la)})$. Then the \abac for $\mu_\la$ is obtained from the \abac for $\la$ by moving the last bead down one position. It is easy to see that $\mu_\la\in O_\la$ and $\mu_\la\setminus\la$ has \conf $X_i$.

\begin{eg}
Let $m=37$, $l=30$ and $p=n=7$. Consider the partition $\la=(12,8,5,4,1)\vdash30$. Then $\mu_\la=(19,8,5,4,1)$ and $\mu_\la\setminus\la$ has $7$-shape $X_4$. %(The \abacs of the following three examples all have charge 14.)
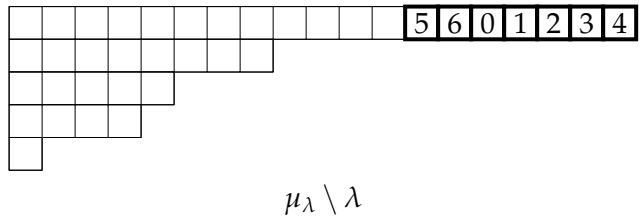
\[
\begin{array}{c}
	%\abacus(lmmmmmr,bbbbbbb,bbnbnnn,bnbnnnb,nnnnnnn,nnnn*nn)
	%\\
	%\mu_\la
	%\\[12pt]
	%\abacus(lmmmmmr,bbbbbbb,bbnbnnn,bnbnnnb,nnnnbnn,nnnnnnn)
	%\\
	%\la
	%\\[22pt]
	\begin{tikzpicture}[scale=0.95]
		\tyng(0cm,0cm,19,8,5,4,1)
		\Ylinethick{1.5pt}
		\tgyoung(0cm,0cm,::::::::::::;5;6;0;1;2;3;4,::::::::,:::::,::::,:)
	\end{tikzpicture}
	\\
	\mu_\la\sm\la
\end{array}
\]
\end{eg}

Now take $\mu\in O_\la$ with $\mu\ne\mu_\la$, and call $X$ the \conf of $\mulam$.

\begin{propn}\label{XXi}
$X$ and $X_i$ are $p$-linked.
\end{propn}

\begin{pf}
If a \conncomp of $X$ has hand node $\bxi$ then we are done because $\ch{X_i}$ and $\ch{X}$ share the term $(i+1,\dots,i)$. Suppose that this is not the case, and let $j\in\zpz$ be such that $\bxj$ is the hand node of the \conncomp of $X$ containing $\bxi$. We firstly observe that $\mu_1=\la_1$ because otherwise the nodes $(1,\la_1+1),\dots,(1,\mu_1)$ of $\mu$ would form a \conncomp of $X$ with $\bxipo$ as foot node and this would imply that $X$ has a \conncomp whose hand node is $\bxi$.

Now we can modify $\mu$ by removing the nodes $\bxipo,\dots,\bxj$ and putting them back at the end of the first row. Let $\tilde\mu\vdash m$ be the resulting partition. By construction $\mathrm{core}_p(\tilde\mu)=\mathrm{core}_p(\mu)$, then $\tilde\mu\setminus\la$ is in $B$. If $\tilde X$ denotes the \conf of $\tilde\mu\setminus\la$, then $\tilde X$ is $p$-linked with both $X_i$ and $X$: $\ch{X}$ and $\ch{\tilde X}$ share the term $(i+1,i+2,\dots,i)$ while $\ch{X_i}$ and $\ch{\tilde X}$ share $(j+1,j+2,\dots,j)$. So $X$ and $X_i$ are $p$-linked.

%\mattnote{It seemed much easier to do this directly on Young diagrams, rather than the abacus. Also easier notation to use $j$ in place of $i+j$.}
% Let $x_k\in\mathbb{Z}$ denote the position of the node $\bxk$ of $X$ in the abacus, for $0\le k\le p-1$. Then $x_{i+j}$ is a star bead and $x_i$ is a space in $\Ab{\mu}$. Also, since $\mu_1=\la_1$, position $\mu_1-1+K$ is a bead and $\mu_1-1+j+K$ is a space (where $K$ is the charge, i.e. the number of beads). Then we modify $\mu$ by moving the bead at $x_{i+j}$ to the space at $x_i$, and moving the bead at $\mu_1+K-1$ to $\mu_1+k-1+j$. We let $\tilde\mu\vdash m$ be the resulting partition. , $\tilde\mu\setminus\la$ is in $B$ and its \conf $\tilde X$ is $p$-linked with both $X_i$ and $X$: $\ch{X}$ and $\ch{\tilde X}$ share the term $(i+1,i+2,\dots,i)$ while $\ch{X_i}$ and $\ch{\tilde X}$ share $(i+j+1,i+j+2,\dots,i+j)$ (where $i+1$, $i+j$ and $i+j+1$ are all read modulo $p$).
\end{pf}

\begin{eg}
Let $m,l,p,n$ and $\la$ be as in the previous example. Let $\mu=(12^2,5,4^2)\in O_\la$. The \skewp $\mulam$ belongs to the \conjbl $B$ of $\calc_{30,37}^k$ characterised by the $7$-core $\gamma\vcentcolon=\mathrm{core}_7(B)=(1^2)\vdash2$. Call $X$ the $7$-shape of $\mulam$.
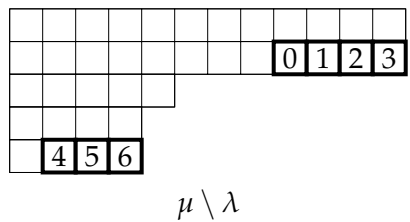
\[
\begin{array}{c}
\begin{tikzpicture}[scale=0.95]
\tyng(0cm,0cm,12,12,5,4,4)
\Ylinethick{1.5pt}
\tgyoung(0cm,0cm,::::::::::::,::::::::;0;1;2;3,:::::,::::,:;4;5;6)
\end{tikzpicture}
\\
\mulam
\end{array}
\]
Retracing the notation of \cref{XXi}, we have that $i=\widebar4$ (because $\la_1-1=11\equiv4\ppmod7$ and then $\la\in L_4$) and $j=\widebar6$ because the node $\young(4)$ belongs to the \conncomp of $X$ having hand node $\young(6)$. So we find $\tilde\mu$ by moving the nodes $\young(56)$ to the end of the first row, giving $\tilde{\mu}=(14,12,5,4,2)\vdash37$.
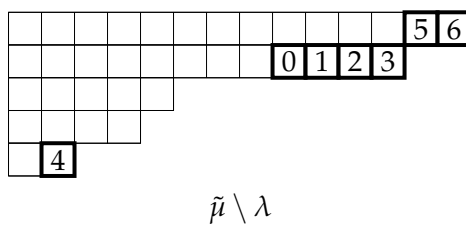
\[
\begin{array}{c}
\begin{tikzpicture}[scale=0.95]
\tyng(0cm,0cm,14,12,5,4,2)
\Ylinethick{1.5pt}
\tgyoung(0cm,0cm,::::::::::::;5;6,::::::::;0;1;2;3,:::::,::::,:;4)
\end{tikzpicture}
\\
\tilde\mu\sm\la
\end{array}
\]
Now the $7$-shape $\tilde{X}$ of $\tilde{\mu}\setminus\la$, is $7$-linked with both $X$ and $X_4$.
\end{eg}

\begin{rmk}
If $\la_1,\la_2\in L_i$, then $\mu_{\la_1}\setminus\la_1$ and $\mu_{\la_2}\setminus\la_2$ have the same \conf $X_i$. By the previous proposition we deduce that, for every $\mu_1\in O_{\la_1}$ and $\mu_2\in O_{\la_2}$, the \confs of $\mu_1\setminus\la_1$ and $\mu_2\setminus\la_2$ are $p$-linked.
\end{rmk}

We are now very close to the end of this section. The next result will link \confs $X_i,X_j$ for all pairs of $i,j\in\zpz$. Recall that $c=\ol{\ga_1-1}$, and observe that $L_c$ is non-empty: for example, it contains the partition $\ol\ga=(\ga_1+wp,\ga_2,\ga_3,\dots)$. In particular this shows that $B$ contains at least one \skewp with \conf $X_c$.

%Below we will need our \abacs to have finitely many beads. As explained in \cref{seccore}, we can trucate the abacus displays fixing a starting position $N$ that is small enough for our purposes. So fix an integer $N\le-m-1$ such that $N\equiv0\ppmod p$. All the \abacs in the proof of \cref{XiXc} are upper truncated and $N$ is the upper left-side position.

\begin{propn}\label{XiXc}
Let $i\in\zpz$ such that $L_i\ne\emptyset$. Then $X_i$ and $X_c$ are $p$-linked.
\begin{pf}
Let $\la\in L_i$ and consider the \skewps $\mu_\la\setminus\la$ and $\mu_{\ol{\gamma}}\setminus\ol{\gamma}$ in $B$ having \confs $X_i$ and $X_c$, respectively. For convenience we suppose that $c=\widebar0$, though the proof can easily be modified for other values of $c$.

If $i=\widebar0$, the result follows by the previous remark.

Suppose that $i\ne\widebar0$. First of all we observe that in the \abac of $\gamma$, the leftmost runner, i.e.\ runner 0, has more beads than any other runner: this is because the last bead on the abacus occurs on this runner and (because $\ga$ is a $p$-core) the beads are at the topmost positions of their runners. So, following the notation of \cref{seccore}, we have that for every $j\in\zpz\setminus\{\widebar0\}$, 
\begin{equation}
b_0(\gamma)>b_j(\gamma).
\end{equation}
The \abac of $\la$ is obtained by lowering down beads in $\Ab{\gamma}$ a total amount of $w$ times. Now recall the $p$-quotient $(\la^{(0)},\dots,\la^{(p-1)})$ of $\la$. The fact that $\la\in L_i$ implies that
\begin{equation}
{\la^{(i)}}_1\ge{\la^{(0)}}_1+b_0(\la)-b_i(\la),
\end{equation} where ${\la^{(i)}}_1$ and ${\la^{(0)}}_1$ denote the number of spaces above the lowest beads in the $i$-th and 0-th runners of $\Ab{\la}$, respectively. Recalling that $b_j(\la)=b_j(\gamma)$ for all $j\in\zpz$, we combine $(1)$ and $(2)$ getting
\begin{equation}
{\la^{(i)}}_1>{\la^{(0)}}_1.
\end{equation}
By construction, if $(\mu_\la^{(0)},\dots,\mu_\la^{(p-1)})$ denotes the $p$-quotient of $\mu_\la$, then $(\mu_\la^{(j)})_1={\la^{(j)}}_1$ for every $j\in\zpz\setminus\{i\}$ and
\begin{equation}
{(\mu_\la^{(i)})}_1={\la^{(i)}}_1+1.
\end{equation}
Comparing $(3)$ and $(4)$ we find that \[
{(\mu_\la^{(i)})}_1>{(\mu_\la^{(0)})}_1+1.
\]
This inequality implies that $\mu_\la$ has at least two addable $i$-hooks at runner 0. We consider one of them, say at position $a_1\equiv 0\ppmod p$, avoiding the one that affects the space immediately above the lowest bead in the $i$-th runner of $\Ab{\mu_\la}$, if necessary. Now let $1\le k_1\le i$ be the least integer such that $a_1+i-k_1$ is a bead. If $k_1<i$, by $(1)$ and the fact that $b_j(\mu_\la)=b_j(\gamma)$ for all $j\in\zpz$, $\mu_\la$ has an addable $(i-k_1)$-hook at runner 0, say at position $a_2$. Then we start again the process from position $a_2$, letting $k_2$ be minimal such that position $a_2+i-k_1-k_2$ is a bead. It is clear that this procedure terminates at a step $s$ such that $i=\sum_{1\le u\le s}k_u$. Then for $1\le v\le s$ we have positions $c_v\vcentcolon=a_v+i-\sum_{1\le u\le v}k_u$ and $d_v\vcentcolon=a_v+i-\sum_{1\le u\le v-1}k_u$ such that $c_v$ is a bead and $d_1,\dots,d_v$ are all spaces in $\Ab{\mu_\la}$. Then we modify $\Ab{\mu_\la\setminus\la}$ as follows (we denote $C$ the charge of the abacus display):
\begin{itemize}
\item move the right half-bead at $x\vcentcolon=\la_1+p-1+C$ to the space at $x-i$,
\item for every $1\le v\le s$, move the right half of the bead at $c_v$ to the space $d_v$.
\end{itemize}

An illustrative abacus is drawn in \cref{xlinkfig} (only runners from 0 to $i$ are shown).

\begin{figure}[h]
\[
\begin{tikzpicture}[label distance=-4.5pt]
\foreach\g in{4.5}{
	\foreach\x in {0,1.5,3,4.5}\draw(\x,-.5)--++(0,6)++(\g+4.5,0)--++(0,-6);
	
	%beads
	\draw(0,3)node(1){\ };\draw(1)node{$\abacus(b)$};
	\draw(0,4)node(3){\ };\draw(3)node{$\abacus(b)$};
	\draw(0,5)node(4){\ };\draw(4)node{$\abacus(b)$};
	\draw(6+\g,5)node[label=below left:$c_1$](c1){\ };\draw(c1)node{$\abacus(b)$};\begin{scope}[scale=.6]\bluertbll{10+5*\g/3}{25/3}\end{scope}
	\draw(0,2)node[label=below left:$c_s$](cs){\ };\draw(cs)node{$\abacus(b)$};\begin{scope}[scale=.6]\bluertbll0{10/3}\end{scope}
	\draw(3,4)node[label=below left:$c_{s-1}$](cs-1){\ };\draw(cs-1)node{$\abacus(b)$};\begin{scope}[scale=.6]\bluertbll5{20/3}\end{scope}
	\draw(4.5+\g,2)node[label=below left:$c_2$](c2){\ };\draw(c2)node{$\abacus(b)$};\begin{scope}[scale=.6]\bluertbll{7.5+5*\g/3}{10/3}\end{scope}
	
	%spaces
	\draw(3,2)node[label=below right:$d_s$](ds){\ };\draw(ds)node{$\abacus(n)$};
	\draw(9+\g,5)node[label=below right:$d_1$](d1){\ };\draw(d1)node{$\abacus(n)$};
	\draw(7.5+\g,5)node(2){\ };\draw(2)node{$\abacus(n)$};
	\draw(6+\g,2)node[label=below right:$d_2$](d2){\ };\draw(d2)node{$\abacus(n)$};
	\draw(7.5+\g,0)node(a){\ };\draw(a)node{$\abacus(n)$};
	\draw(6+\g,0)node(b){\ };\draw(b)node{$\abacus(n)$};
	\draw(4.5+\g,0)node(c){\ };\draw(c)node{$\abacus(n)$};
	\draw(4.5,0)node(d){\ };\draw(d)node{$\abacus(n)$};
	\draw(0,0)node[label=below left:$x-i$](f){\ };\draw(f)node{$\abacus(n)$};
	\draw(1.5,0)node(e){\ };\draw(e)node{$\abacus(n)$};
	\draw(3,0)node(g){\ };\draw(g)node{$\abacus(n)$};
	\draw(1.5,2)node(h){\ };\draw(h)node{$\abacus(n)$};
	\draw(4.5,4)node(i){\ };\draw(i)node{$\abacus(n)$};
	\draw(4.5+\g,3)node[label=below right:$d_3$](d3){\ };\draw(d3)node{$\abacus(n)$};
	
	%stars
	\draw(9+\g,0)node[label=below right:$x$](s){\ };%\draw(s)node{$\abacus(M)$};
	\begin{scope}[scale=.6]\redrtbll{15+5*\g/3}0\end{scope}
	
	%dots
	\draw(.75,0)node{\dots};
	\draw(.75,2)node{\dots};
	\draw(2.25,0)node{\dots};
	\draw(2.25,2)node{\dots};
	\draw(3.75,0)node{\dots};
	\draw(3.75,4)node{\dots};
	\draw(5.25,0)node{\dots};
	\draw(5.25,4)node{\dots};
	\draw(3.75+\g,0)node{\dots};
	\draw(3.75+\g,3)node{\dots};
	\draw(5.25+\g,0)node{\dots};
	\draw(5.25+\g,2)node{\dots};
	\draw(6.75+\g,0)node{\dots};
	\draw(6.75+\g,5)node{\dots};
	\draw(8.25+\g,0)node{\dots};
	\draw(8.25+\g,5)node{\dots};

	%arrows
	\draw[->,thick,red](s)to[out=165,in=15](f);
	
	\draw[->,thick,\myblue](c1)to[out=30,in=150](d1);
	\draw[->,thick,\myblue](cs)to[out=30,in=150](ds);
	\draw(cs-1)++(2,.6)coordinate(cc);
	\draw[->,thick,\myblue](cs-1)to[out=30,in=180](cc);
	\draw[->,thick,\myblue](c2)to[out=30,in=150](d2);
	\draw[->,thick,\myblue](d3)+(-1,.3)to[out=0,in=150](d3);
}
\end{tikzpicture}
\]
\caption{$\Ab{\mu_\la\sm\la}$.}\label{xlinkfig}
\end{figure}

The resulting partition $\hat\mu\vdash m$ is such that $\hat\mu\setminus\la\in B$ and has \conf $\hat X$ that is $p$-linked with both $X_i$ and $X_c$: the \forchs $\ch{X_i}$ and $\ch{\hat X}$ share the term $(i+1,\dots,i)$ while $\ch{X_c}$ and $\ch{\hat X}$ share the term $(\widebar1,\widebar2,\dots,\widebar{p-1},\widebar0)$. This ends the proof.

\end{pf}
\end{propn}

\begin{eg}
Let $B$ the \conjbl of $\calc_{30,37}^{k}$ from the previous examples. As already pointed out, following the notation, $\gamma=(1^2)$ and then $c=\gamma_1-1\equiv0\ppmod7$. We see that $\ol\gamma=(29,1)$ and $\mu_{\ol\gamma}=(36,1)$.
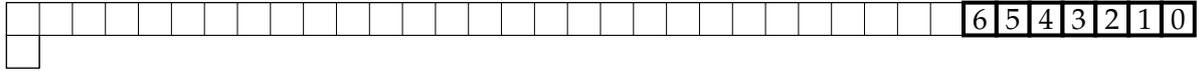
\[
\begin{array}{c}
\begin{tikzpicture}[scale=0.95]
\tyng(0cm,0cm,36,1)
\Ylinethick{1.5pt}
\tgyoung(0cm,0cm,:::::::::::::::::::::::::::::;6;5;4;3;2;1;0,:)
\end{tikzpicture}
\\
\mu_{\ol\gamma}\sm\ol\gamma
\end{array}\]
We show how to modify the \skewp $\mu_\la$ from the previous example, to find a $7$-shape that is $7$-linked with both $X_4$ and $X_0$.

We first draw the \abac of $\ga$ (all \abacs in the example have charge $C=14$). 
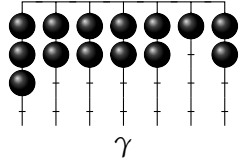
\[
\begin{array}{c}
\abacus(lmmmmmr,bbbbbbb,bbbbbnb,bnnnnnn,nnnnnnn)
\\
\ga
\end{array}
\]
Note that as remarked in the proof of \cref{XiXc}, the number of beads in runner 0 of $\gamma$, $b_0(\gamma)=3$, is bigger than the number of beads in any other runner of $\Ab{\gamma}$. 

Our proof points out that $\mu_\la$ has at least two addable 4-hooks at runner 0, actually these are at positions 7 and 14. %Take $a_1=14$ and $d_1=18$. We have that $k_1=2$ and so $c_1=16$. Now we consider the addable 2-hook of $\mu_\la$ at position $a_2=7$. This gives $d_2=9$ and consequently $k_2=1$ and $c_2=8$. Finally we take the addable node at runner 0 of $\mu_\la$ at position $a_3=14$. We have that $d_3=15$, giving $k_3=1$ and $c_3=14$ (then $s=3$).
Choosing $a_1=14$ and following the construction in the proof of \cref{XiXc}, we modify $\Ab{\mu_\la\setminus\la}$ as shown in the picture. The resulting \skewp $\hat\mu\setminus\la=(15,8,7,5,1^2)\setminus(12,8,5,4,1)$ has the desired $7$-shape.

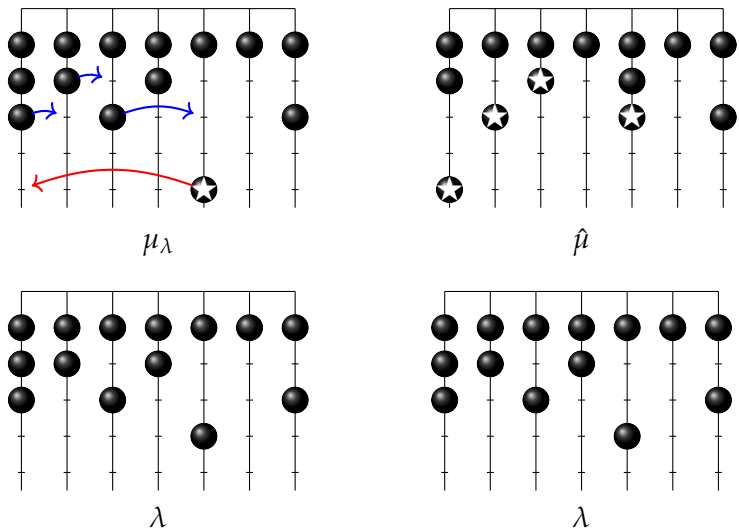
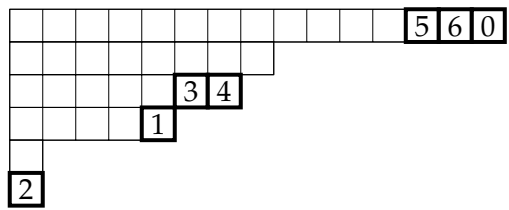
\[
\begin{array}{c@{\qquad\qquad}c}
\begin{tikzpicture}[scale=.6]
\bound
\draw(0,4.5*\yoy)--++(6,0);
\foreach\x in{0,1,2,3,4,5,6}{\bll\x4}
\foreach\x in{2,4,5,6}{\spce\x3}\foreach\x in{0,1,3}{\bll\x3}\bluertbll13
\foreach\x in{1,3,4,5}{\spce\x2}\foreach\x in{0,2,6}{\bll\x2}\bluertbll02\bluertbll22
\foreach\x in{0,1,2,3,5,6}{\spce\x1}\lfbll41
\foreach\x in{0,1,2,3,5,6}{\spce\x0}\redrtbll40
\draw(4,0*\yoy)node(b){\ };
\draw(0,0*\yoy)node(a){\ };
\draw[->,red,thick](b)to[out=160,in=20](a);
\draw(0,2*\yoy)node(d){\ };
\draw(1,2*\yoy)node(c){\ };
\draw[<-,\myblue,thick](c)to[out=160,in=20](d);
\draw(2,2*\yoy)node(e){\ };
\draw(4,2*\yoy)node(f){\ };
\draw[<-,\myblue,thick](f)to[out=160,in=20](e);
\draw(1,3*\yoy)node(g){\ };
\draw(2,3*\yoy)node(h){\ };
\draw[<-,\myblue,thick](h)to[out=160,in=20](g);
\end{tikzpicture}
&
\begin{tikzpicture}[scale=.6]
\bound
\draw(0,4.5*\yoy)--++(6,0);
\foreach\x in{0,1,2,3,4,5,6}{\bll\x4}
\foreach\x in{4,5,6}{\spce\x3}\bll03\bll33\lfbll13\rtbll23
\foreach\x in{3,5}{\spce\x2}\bll62\lfbll02\lfbll22\rtbll12\rtbll42
\foreach\x in{0,1,2,3,5,6}{\spce\x1}\lfbll41
\foreach\x in{4,1,2,3,5,6}{\spce\x0}\rtbll00
\end{tikzpicture}
\\
\mu_\la\sm\la&\hat\mu\sm\la
\end{array}
\]
	
\[
\begin{array}{c}
\begin{tikzpicture}[scale=0.95]
\tyng(0cm,0cm,15,8,7,5,1,1)
\Ylinethick{1.5pt}
\tgyoung(0cm,0cm,::::::::::::;5;6;0,::::::::,:::::;3;4,::::;1,:,;2)
\end{tikzpicture} 
\\
\hat\mu\sm\la
\end{array}
\]
\end{eg}

We have finally gained the key result of this section.

\begin{cory}\label{beltthm}
	Suppose $X,Y\in\calb$. Then $X$ and $Y$ are $p$-linked.
	\begin{pf}
		By \cref{prop2belt}, we can reduce to the case when $\dist XY=0$. Moreover, \cref{beltequiv} tells us that we are done if $\arrX\cup\arrY$ is not an oriented cycle. In the opposite case, the description of the remaining cases together with \cref{XXi,XiXc} ends the proof.
	\end{pf}
\end{cory}

Now, \cref{beltthm,beltequiv} implies that the decomposition matrix $D_B$ of the belt block $B$ is connected. As remarked along the treatment, this is enough to state our main result.

\begin{cory}\label{beltcor}
	\cref{mainconj} holds for belt blocks.
\end{cory}

\end{document}